\documentclass[11pt,a4paper]{article}
 
\usepackage[utf8]{inputenc}
\usepackage[english]{babel}
\usepackage[T1]{fontenc}
\usepackage{amsmath}
\usepackage[hidelinks]{hyperref}
\usepackage{amsfonts}
\usepackage{amssymb}
\usepackage{amsthm}
\usepackage{tikz-cd}
\usepackage{mathtools}
\usepackage{multicol}
\usepackage{authblk}
\usepackage{changepage}
\usepackage[left=2cm,right=2cm,top=3cm,bottom=3cm, footskip=1.5cm]{geometry}
\usepackage{setspace}
\usepackage[style=alphabetic, backend=biber, doi=false, isbn=false, url=false]{biblatex}
\usepackage{fancyhdr}
\AtEveryBibitem{
  \clearfield{note}
  \ifentrytype{book}{\clearfield{pages}}
}

\author{Gustave Billon}
\title{Moduli Spaces of Marked Branched Projective Structures on Surfaces}

\DeclareMathOperator{\PGL}{PGL}

\DeclareMathOperator{\Ad}{Ad}

\DeclareMathOperator{\CP}{\mathbb CP}

\DeclareMathOperator{\Hom}{Hom}

\DeclareMathOperator{\Sym}{Sym}
\DeclareMathOperator{\Part}{Part}

\DeclareMathOperator{\red}{red}

\DeclareMathOperator{\pr}{pr}

\DeclareMathOperator{\rel}{rel}

\DeclareMathOperator{\Aut}{Aut}
\DeclareMathOperator{\C}{\mathbb C}
\DeclareMathOperator{\R}{\mathbb R}

\DeclareMathOperator{\N}{\mathbb N}
\DeclareMathOperator{\comma}{, }
\DeclareMathOperator{\Teich}{Teich}
\DeclareMathOperator{\Hol}{Hol}

\DeclareMathOperator{\PSL}{PSL}
\DeclareMathOperator{\Div}{div}
\DeclareMathOperator{\BC}{bc}
\DeclareMathOperator{\BPS}{bps}
\DeclareMathOperator{\Image}{Im}

\newtheorem{statement}{}[subsection]

\theoremstyle{definition}
\newtheorem{definition}[statement]{Definition}
\newtheorem*{definition*}{Definition}

\theoremstyle{theorem}
\newtheorem{theorem}[statement]{Theorem}
\newtheorem{proposition}[statement]{Proposition}
\newtheorem{corollary}[statement]{Corollary}
\newtheorem{lemma}[statement]{Lemma}
\newtheorem{theoremx}{Theorem}

\theoremstyle{remark}
\newtheorem{remark}[statement]{Remark}

\begin{document}

\maketitle

\begin{abstract}
We show that the moduli space $\mathcal P_g(n)$ of marked branched projective structures of genus $g$ and branching degree $n$ is a complex analytic space.
In the case $g \ge 2$ we show that $\mathcal P_g(n)$ is of dimension $6g-6+n$ and we characterize its singular points in terms of their monodromy.
We introduce a notion of branching class, that is an infinitesimal description of branched projective structures at the branched points.
We show that the space $\mathcal A_g(n)$ of marked branching classes of genus $g$ and branching degree $n$ is a complex manifold.
We show that if $n < 2g-2$ the space $\mathcal P_g(n)$ is an affine bundle over $\mathcal A_g(n)$, while if $n > 4g-4$ $\mathcal P_g(n)$ is an analytic subspace of $\mathcal A_g(n)$.
\end{abstract}

\section*{Introduction}
\addcontentsline{toc}{section}{Introduction}

\subsection*{Definition and Examples}

Holomorphic projective structures were introduced at the end of the nineteenth century in relation to linear differential equations of order $2$. They were used in particular by Poincaré, as an analogous for curves of genus at least $2$ of elliptic functions, in his proof of the uniformization theorem of Riemann surfaces. 
They were then thoroughly studied throughout the twentieth century, in particular by Gunning (see for instance \cite{gunning}).

In \cite{mandelbaum_1}, Mandelbaum introduced a notion of \emph{branched} projective structures, that can be seen as a generic generalization of the concept of projective structure.

Let us give a precise definition.
Fix an oriented surface $S$.

\begin{definition*}
A \emph{branched projective structure} on $S$ is the datum of
\begin{itemize}
\item An open cover $(U_i)_{i \in I}$ of $S$
\item For each $i \in I$, a locally ramified covering $f_i : U_i \to \CP^1$ such that on the intersections $U_i \cap U_j$, one has $f_i = g_{ij} \circ f_j$, where $g_{ij} \in \PGL(2, \C)$.
\end{itemize}
The number of ramifications of the maps $f_i$, counted with multiplicity, is called the \emph{branching degree} of the branched projective structure.
\end{definition*}

In the case where the branching degree is zero, one gets \emph{projective structures}, see \cite{dumas} and \cite{loray-marin} for overview papers.
Here are a few examples of branched projective structures :

\paragraph*{Tautological projective structure}
The Riemann sphere $\CP^1$ is endowed with a trivial projective structure, of branching degree zero, given by the indentity map.

\paragraph*{Uniformizing projective structure}
Suppose the surface $S$ is closed, and let $X$ be a Riemann surface with underlying differential surface $S$.
Let $\pi : \widetilde X \to X$ be the universal covering of $X$. 
By the uniformization theorem, $\widetilde X \subset \CP^1$ and $\Aut(\pi) \subset \PGL(2, \C)$.
As a consequence, the local inverses of the map $\pi$ are charts of a projective structure on $X$, without ramification point, called the \emph{uniformizing projective structure} of $X$.

\paragraph*{Pullback by a ramified covering}
Let $S'$ be another oriented surface endowed with a branched projective structure $(V_i, h_i)_{i \in I}$, and let $\phi : S \to S'$ be a ramified covering.
Then $S$ inherits a \emph{pullback} projective structure : the open covering is $\left(\phi^{-1}(V_i)\right)_{i \in I}$ and the family of charts is $(h_i \circ \phi)_{i \in I}$.
In particular, a ramified covering of the Riemann sphere can be seen as a branched projective structure.

\paragraph*{Curve not preserved by a foliation}
There is a notion of projective structure \emph{transverse to a foliation}, see \cite{scardua}.
Precisely, let $M$ be a complex manifold and $\mathcal F$ a codimension $1$ holomorphic foliation on $M$.
A projective structure transverse to $\mathcal F$ is an open cover $(\mathcal U_i)_{i \in I}$ of $M$ and holomorphic submersions $s_i : \mathcal U_i \to \CP^1$, constant on the leaves of $\mathcal F$, such that on the intersections $\mathcal U_i \cap \mathcal U_j$ one has $s_i = g_{ij} \circ s_j$ with $g_{ij} \in \PGL(2, \C)$.
Now if $X \subset M$ is a complex curve that is not invariant by $\mathcal F$, then $X$ inherits a branched projective structure, whose branching degree is the number of tangencies between $X$ and $\mathcal F$.

\paragraph*{Linear differential equations of order $2$}
Projective structures with branching degree zero can be obtained from equations of the form $a(x)y''(x) + b(x)y'(x) + c(x)y(x) = 0$ where $x$ is a local coordinate, $y$ is an unknown holomorphic function and $a, b, c$ are holomorphic functions, with $a$ nonvanishing.
Given a basis $(y_1, y_2)$ of local solutions, the quotient $\frac{y_1}{y_2}$ is a local biholomorphism from an open subset of $X$ to $\CP^1$.
If $(w_1, w_2)$ is another basis of local solutions, then there exists $\alpha, \beta, \gamma, \delta \in \C$ with $\alpha \delta - \beta \gamma \neq 0$ such that $w_1 = \alpha y_1 + \beta y_2$ and $w_2 = \gamma y_1 + \delta y_2$.
Thus there exists a Möbius transformation $g = \left(\begin{smallmatrix} \alpha & \beta \\ \gamma & \delta \end{smallmatrix}\right)$ such that $\frac{w_1}{w_2} = g \circ \left(\frac{y_1}{y_2}\right)$.
In other words, the equation defines a projective structure on $X$.

This description of projective structures gives rise to a notion of \emph{meromorphic projective structure}, corresponding to order $2$ differential equations with meromorphic coefficients.
This notion was studied in particular in \cite{allegretti-bridgeland}, \cite{gupta-mj}, \cite{gupta-mj_2} and \cite{serandour}.
Branched projective structures are very special meromorphic projective structures, that are usually excluded from the studies of general meromorphic projective structures.

Note that the notion of linear differential equations on Riemann surfaces is formalized in the notion of \emph{opers}, see \cite{beilinson-drinfeld}.
In this language, projective structures are $\PGL(2, \C)$-opers. See \cite{frenkel} for the role of opers in the Langlands program.
This thesis thus deals with \emph{branched} $\PGL(2, \C)$-opers. 
The notion of branched opers has been investigated in \cite{frenkel-gaitsgory} and \cite{biswas-dumitrescu-heller}.
See also \cite{biswas-dumitrescu-gupta} for the link between branched projective structures and logarithmic connections.

\paragraph*{Topological constructions}
Branched projective structures can also be modified by cut-and-paste techniques, some preserving the branching degree, such as \emph{grafting} or {moving branch points}, others changing the branching degree, such as \emph{bubbling}.
See for instance \cite{dumas}, \cite{gallo-kapovich-marden} or \cite{calsamiglia-deroin-francaviglia}.
These cut-and-paste techniques allow to construct branched projective structures of any degree.

\subsection*{Main Result}

Suppose the surface $S$ is closed, of genus $g$.
Fix a nonnegative integer $n \in \N$.
There is a notion of \emph{marked} branched projective structure, similar to the notion of marked complex structure on $S$ (see \ref{subsection_families_bps} for precise definitions).
We are interested in the following space :
\begin{equation}
\mathcal P_g(n) = \frac{\{\text{marked branched projective structures on $S$}\}}{\{\text{isomorphisms}\}}
\end{equation}
Given a marked branched projective structure $(U_i, f_i)_{i \in I}$ with $f_i = g_{ij} \circ f_j$, $g_{ij} \in \PSL(2, \C)$, the changes of charts $g_{ij}$ are holomorphic maps, thus the (ramified) atlas $(U_i, f_i)_{i \in I}$ defines a marked complex structure on $S$.
Denote by $\mathcal T_g$ the Teichmüller space for genus $g$, i.e. the space of isomorphism classes of marked Riemann surfaces of genus $g$.
One has a forgetful map
\begin{equation}
\Teich_g(n) : \mathcal P_g(n) \to \mathcal T_g(n)
\end{equation}

Moreover, the family $\left(g_{ij}\right)_{i, j \in I}$ is a constant cocycle with values in $\PSL(2, \C)$, and thus defines a flat $\CP^1$-bundle on $X$, to which is associated a holonomy representation $\rho \in \Hom\left(\pi_1(S), \PSL(2, \C)\right)$, defined up to conjugation with an element of $\PSL(2, \C)$.
Thus there is a holonomy map
\begin{equation}
\Hol_g(n) : \mathcal P_g(n) \to \Hom\left(\pi_1(S), \PSL(2, \C)\right) / \PSL(2, \C)
\end{equation}
The holonomy of branched projective structures has been extensively studied, see \cite{gallo-kapovich-marden}, \cite{calsamiglia-deroin-francaviglia}, \cite{le_fils}.

In the unbranched case, i.e. in the case $n=0$, if $g \ge 2$, it is well-known (\cite{gunning}, \cite{hejhal}, \cite{hubbard}) that the space $\mathcal P_g(0)$ is a smooth analytic space.
Moreover the fiber of $\Teich_g(0)$ of a marked complex curve $X \in \mathcal T_g$ is an affine space, directed by the space of global holomorphic quadratic differentials on $X$, $H^0\left(X, K_X^{\otimes 2}\right)$, and the holonomy map $\Hol_g(0)$ is a local biholomorphism.
This very nice structure of the moduli space of projective structures is one of the main reasons why they are extensively used.

The main result of this paper is an analog of these properties in the branched case (see theorems \ref{theorem_moduli_space_bps}, \ref{theorem_singular_points} and \ref{theorem_holonomy} for more precise statements).

\begin{theoremx}\label{theorem_introduction}
The space $\mathcal P_g(n)$ is an analytic space of dimension $6g-6+n$.
Morever, if $g \ge 2$, one has :
\begin{itemize}
\item[(i)] A point $p \in \mathcal P_g(n)$ is singular if and only if its holonomy representations $\Hol_g(n)(p)$ are abelian and fix a point in $\CP^1$.
\item[(ii)] The map $\Hol_g(n)$ is a holomorphic submersion over branched projective structures with nonelementary holonomy.
%\item[(ii)] Let $\Hom^*\left(\pi_1(S), \PSL(2, \C)\right) \subset \Hom\left(\pi_1(S), \PSL(2, \C)\right)$ be the dense open subset of nonabelian representations that do not admit a $2$-point orbit in $\CP^1$. Then $\Hol_g(n)$ is a submersion in restriction to the preimage of $\Hom^*\left(\pi_1(S), \PSL(2, \C)\right) / \PSL(2, \C)$.
\end{itemize}
\end{theoremx}

Recall that a representation $\rho \in \Hom\left(\pi_1(S), \PSL(2, \C)\right)$ is said to be \emph{elementary} if its image is an elementary subgroup of $\PSL(2, \C)$.
A subgroup of $\PSL(2, \C)$ is elementary if its action on the hyperbolic space $\mathbb H^3$ admits a finite orbit. 
Namely, elementary subgroups of $\PSL(2, \C)$ are the ones that fix a point in $\CP^1$, the ones that fix a pair of points in $\CP^1$ (conjugated to a subgroup of $\{z \mapsto  az | a \in \C^*\} \cup \{z \mapsto \frac a z | a \in \C^*\}$), and the ones that fix a point in the interior of $\mathbb H^3$ (spherical subgroups).
Note that the condition of having nonelementary holonomy in point (ii) is not optimal, see theorem \ref{theorem_holonomy}.

Theorem \ref{theorem_introduction} answers a question asked in the appendix of \cite{calsamiglia-deroin-francaviglia}, where it is shown that the fibers of $\Hol_g(n)$ over nonelementary representations are smooth analytic spaces.
This also extends results obtained by Mandelbaum in his papers introducing branched projective structures, see \cite{mandelbaum_1}, \cite{mandelbaum_2}, \cite{mandelbaum_3}.

The key point in the proof of theorem \ref{theorem_introduction} is what we call \emph{branching classes}.
Their existence is due to the fact that, unlike local biholomorphisms, local branched coverings are not all obtained from one another by postcomposition with a local biholomorphism.
See section \ref{section_branching_classes} for details.

\subsection*{Structure of the Paper}

In section \ref{section_branching_classes}, we introduce the notion of branching class over a curve with a divisor.
We investigate the reasons why branching classes are a useful tool for our purpose. 
These reasons are the very nice structure of the space of branching classes over a fixed curve with divisor, and the very nice structure of the space of branching classes with fixed branching class.

In section \ref{section_families_branched_projective_structures}, we define the moduli spaces of branched projective structures $\mathcal P_g(n)$, and we exhibit its analytic structure.
We consider in turn the analytic structures of spaces of marked curves (the Teichmüller spaces), the spaces of marked curves with divisors and the spaces of marked branching classes, before being able to endow $\mathcal P_g(n)$ with an analytic structure.
We also have to introduce a technical notion of \emph{restricted} branched projective structures, that we use as an intermediate between branching classes and branched projective structures.
We emphasize in this section the universal properties of the moduli spaces we consider, working with \emph{relative} branched projective structures, or equivalently families of such structures, in the spirit of \cite{hubbard}.

In section \ref{section_singular_points} we study $\mathcal P_g(n)$ when the genus $g$ is at least $2$.
Most of this section is dedicated to the characterization of the singular points of $\mathcal P_g(n)$, and follows the same structure as section \ref{section_families_branched_projective_structures}, considering in turn the moduli spaces of curves, divisors, branching classes, restricted branched projective structures and finally branched projective structures.
The last subsection is dedicated to the study the holonomy map.
The main tool for this section, in the spirit of \cite{hubbard}, is the Kodaira-Spencer machinery that allows to identify the tangent space of a moduli space at a point with the first cohomology group of infinitesimal automorphisms of this point.

\subsection*{Acknowledgements}

I wish to thank my advisor Sorin Dumitrescu for his help and support.
I also thank Bertrand Deroin and Frank Loray for fruitful conversations.

\subsection*{Notations}

If $X$ is a Riemann surface, $T_X$ denotes the tangent bundle of $X$ and $K_X$ denotes its cotangent bundle, which is also the canonical bundle of $X$.

If $E$ is a holomorphic vector bundle on a complex manifold, we denote also by $E$ its sheaf of holomorphic sections.

Let $z$ be a local coordinate on an open subset $U$ of a Riemann surface, and $\phi = \varphi(z)$ a nonconstant holomorphic function on $U$. The two following differential operators are studied in \cite{gunning} :
\begin{itemize}
\item[(i)] We denote by $[\phi, z]$ the \emph{affine distorsion} of $\phi$ in the coordinate $z$ : $[\phi, z] = \varphi''/\varphi'$
\item[(ii)] We denote by $\{\phi, z\}$ the \emph{schwarzian derivative} of $\phi$ in the coordinate $z$ : $[\phi, z] = \left(\varphi''/\varphi'\right)' - \frac 12 \left(\varphi''/\varphi'\right)^2$
\end{itemize}

\section{Branching Classes}\label{section_branching_classes}

\subsection{The Space of Branching Classes on a Curve with Divisor}

Let $X$ be a Riemann surface, $x \in X$ and $n \in \N$.
Denote by $(X, x)$ the germ of $X$ at $x$, and by $R_x^n$ the complex manifold of $2(n+1)$-jets of $(n+1)$-sheeted ramified coverings from $(X, x)$ to $\CP^1$ :

\begin{equation}\label{eq_space_jets_coverings}
R_x^n = \left\{j^{2(n+1)}_x \phi \left|
\begin{aligned}
&\text{$\phi$ germ at $x$ of holomorphic $(n+1)$-fold} \\
&\text{ramified covering with values in $\CP^1$ }
\end{aligned}
\right. \right\}
\end{equation}

The group of Möbius transformations $\PSL(2, \C)$ acts on $R_x^n$ by postcomposition.

\begin{proposition}
The group $PSL(2, \C)$ acts freely on $R_x^n$.
Moreover, the quotient $\PSL(2, \C) \backslash R_x^n$ is a complex manifold of dimension $n$, isomorphic to $\C^n$.
\end{proposition}

\begin{proof}
Take $z$ a local coordinate on $X$ centered at $x$.
In an affine coordinate $w$ of $\CP^1 = \C \cup \{\infty\}$, an element $j_x^{2(n+1)} \phi \in R_x^n$ can be written :
\begin{equation}
j^{2(n+1)}_x \phi = a_0 + a_{n+1} z^{n+1} + a_{n+2}z^{n+2} + \cdots + a_{2(n+1)}z^{2(n+1)}
\end{equation}
with $a_{n+1} \neq 0$.
Recall that the datum of a Möbius transformation $g \in \PSL(2, \C)$ is equivalent to the datum of its $2$-jet at $a_0$.
If $g = \alpha_0 + \alpha_1 (w-a_0) + \alpha_2 (w-a_0)^2 + O\left((w-a_0)^3\right)$, the action of $g$ on $j^{2(n+1)}_x \phi$ writes
\begin{equation}
\begin{aligned}
&g \cdot j^{2(n+1)}_x \phi = j^{2(n+1)}_x g \circ \phi \\
&= \left(\alpha_0 + a_0\right) + \alpha_1 a_{n+1} z^{n+1} + \alpha_1 a_{n+2} z^{n+2} + \cdots + \alpha_1 a_{2n+1} z^{2n+1} + \left(\alpha_1 a_{2(n+1)} + \alpha_2 a_{n+1}^2\right) z^{2(n+1)}
\end{aligned}
\end{equation}
Note that the coordinate $w$ can be chosen so that neither $\phi(x)$ nor $g(a_0)$ is $\infty$ in the coordinate $w$.
Since $a_{n+1} \neq 0$, the coefficents of degree $0$, $n+1$ and $2(n+1)$ of $g \cdot j^{2(n+1)}_x \phi$ determine the $2$-jet of $g \in \PSL(2, \C)$ at $a_0$, thus $g$ itself, showing that $\PSL(2, \C)$ acts freely on $R_x^n$.
Moreover the orbit of $j^{2(n+1)}_x \phi$ under the action of $\PSL(2, \C)$ is given by the complex parameters $\left(a_{n+2}/a_{n+1}, a_{n+3}/a_{n+1}, \dots, a_{2n+1}/a_{n+1}\right)$, showing that the quotient $\PSL(2, \C) \backslash R_x^n$ is isomorphic to $\C^n$.
\end{proof}

Consider now $D = \sum_{i \in I} n_i x_i$ an effective divisor of degree $n$ on $X$. A \emph{branching class} on $X$ of divisor $D$ is the choice for each point $x_i$ of $D$ of an orbit for the action of $\PSL(2, \C)$ on $R_{x_i}^{n_i}$.
The space $A_X^D$ of branching classes over $(X, D)$ is thus given by
\begin{equation}
A_X^D := \prod_{i \in I} \PSL(2, \C) \backslash R_{x_i}^{n_i}
\end{equation}
If $D$ has finite degree, for instance if $X$ is compact, $A_X^D$ is a complex manifold of dimension $n = \sum_{i=1}^r n_i$.

\begin{proposition}\label{proposition_affine_structure_branching_classes}
The space $A_X^D$ of branching classes over a Riemann surface with divisor $(X, D)$ is an affine space directed by the vector space $H^0\left(X, K_X|_D\right)$ of jets of holomorphic differentials at the points of $D$.
\end{proposition}

\begin{proof}
Let $z$ be la local coordinate defined on an open subset $U \subset X$.
Let $x_0 \in U$ and let $\phi_1 = \varphi_1(z), \phi_2 = \varphi_2(z)$ be nonconstant holomorphic functions on $U$ such that $\phi_1, \phi_2$ are both ramified at $x_0$ with $(n+1)$ branches and have no other ramification points.
Recall that $[\phi_1, z]$ denotes the affine distorsion $\varphi_1''/\varphi_1'$.
A straightforward calculation shows the following (see also \cite{gunning}) :
\begin{itemize}
\item[(i)] The quantity $[\phi_1, z]dz$ is a meromorphic $1$-form on $U$ holomorphic away from $x_0$ and with a simple pole at $x_0$ with residue $n$.
In particular, the difference $[\phi_2, z]dz - [\phi_1, z]dz$ is a holomorphic $1$-form on $U$.
\item[(ii)] The holomorphic $1$-form $[\phi_2, z]dz - [\phi_1, z]dz$ does not depend on the choice of the holomorphic coordinate $z$ on $U$.
\item[(iii)] If $\omega$ is a holomorphic $1$-form on $U$, there exists a neigborhood $V$ of $x_0$ in $U$ and a nonconstant holomorphic function $\phi$ on $V$ whose only ramification point is $x_0$, where it has $(n+1)$ branches, such that $\omega = [\phi, z]dz - [\phi_1, z]dz$.
\item[(iv)] The $n$-jet at $x_0$ $j_{x_0}^n\left([\phi_2, z]dz - [\phi_1, z]dz\right)$ depends only on the $2n+1$-jet at $x_0$ $j_{x_0}^{2n+1}\phi_1$.
\item[(v)] One has $j_{x_0}^n\left([\phi_2, z]dz - [\phi_1, z]dz\right) = 0$ if and only if there exists a local biholomorphism $\gamma$ such that $j_{x_0}^{2n+1}\phi_1 = j_{x_0}^{2n+1}\gamma \circ \phi_1$.
\end{itemize}
These five properties imply that The space $A_U^{nx_0}$ is an affine space, directed by the vector space $H^0\left(X, K_U|_{nx_0}\right)$ of $n$-jets of holomorphic $1$-forms at $x_0$.
The difference between the class of $j_{x_0}^{2(n+1)} \phi_2$ and the class of $j_{x_0}^{2(n+1)} \phi_1$ being given by $j_{x_0}^n\left([\phi_2, z]dz - [\phi_1, z]dz\right) \in H^0\left(X, K_U|_{nx_0}\right)$.

Write $D = \sum_{i \in I} n_i x_i$, and take a family of open sets $(U_i)_{i \in I}$ such that $x_i \in U_i$ and if $i \neq j$, $U_i \cap U_j = \emptyset$.
Then $A_X^D = \prod_{i \in I} A_{U_i}^{n_ix_i}$ is an affine space directed by the vector space $\bigoplus_{i \in I}^r H^0\left(U_i, K_{U_i}|_{n_ix_i}\right) = H^0\left(X, K_X|_D\right)$.
\end{proof}

\begin{remark}
Let $D = \sum_{x \in X} n_x x$ be a divisor on a Riemann surface $X$, $a = (a_x)_{x \in X} \in \break \prod_{x \in X} \PSL(2, \C) \backslash R_{x}^{n_x} = A_X^D$ a branching class over $(X, D)$ and $U \subset X$ an open subset.
We denote by $D|_U = \sum_{x \in U} n_x x$ the divisor $D$ restricted to $U$ and by $a|_U = (a_x)_{x \in U} \in \prod_{x \in U} \PSL(2, \C) \backslash R_{x}^{n_x} = A_U^{D|_U}$ the branching class $a$ restricted to $U$.
\end{remark}

\subsection{Branched Projective Structures and Branching Classes}

Let us first recall the definition of a branched projective structure.

\begin{definition}
A \emph{branched projective atlas} on a Riemann surface $X$ is given by an open cover $(U_i)_{i \in I}$ of $X$, and for each $i \in I$ a nonconstant holomorphic map $f_i : U_i \to \CP^1$ such that for any $i, j \in I$ with $U_i \cap U_j \neq 0$ there is a Möbius transformation $g_{ij} \in \PSL(2, \C)$ such that $f_i|_{U_i \cap U_j} = g_{ij} \circ f_j|_{U_i \cap U_j}$.

Two branched projective atlases are said to be \emph{equivalent} if their union is also a branched projective atlas.

A \emph{branched projective structure} on $X$ is an equivalence class of branched projective atlases.
\end{definition}

Let $X$ be a Riemann surface, let $p$ be a branched projective structure on $X$ and $(U_i, f_i)_{i \in I}$ an atlas for $p$.
On the intersections $U_i \cap U_j$, $f_i$ is obtained from $f_j$ by postcomposition with a Möbius transformation, so the vanishing orders of the differentials $df_i$ and $df_j$ are the same at any point of $U_i \cap U_j$.
Moreover at any point $x \in U_i \cap U_j$, if $n$ is the (possibly zero) vanishing order of $df_i$ (and thus of $df_j$) at $x$, then $f_i$ and $f_j$ both determine elements of $R_x^n$ in \eqref{eq_space_jets_coverings} that are in the same orbit under the action of $\PSL(2, \C)$.

The branched projective structure $p$ thus determines two additional objects on $X$ :
\begin{itemize}
\item[(i)] An effective divisor $D$, called the \emph{branching divisor} of $p$, given by the vanishing locus of the differentials of the charts of $p$ : $D = \sum_{x \in X} n_x x$ where $n_x$ is the vanishing order at $x$ of the charts of $p$ defined at $x$.
\item[(ii)] A \emph{branching class} $a \in A_X^D$ on $X$ of divisor $D$, given by the orbits of the jets of charts of $p$ at the points of $D$.
\end{itemize}

For a branching class $a$ over $(X, D)$, let us denote by $P^a_X$ the set of branched projective structures on $X$ with branching class $a$ (thus with branching divisor $D$).

\begin{proposition}\label{proposition_space_bps_given_bc}
Let $D$ be an effective divisor on a Riemann surface $X$. Let $a \in A_X^D$ be a branching class over $(X, D)$.
\begin{itemize}
\item[(i)] There exists an obstruction $c(a) \in H^1\left(X, K_X^{\otimes 2}(-D)\right)$ such that $c(a)=0$ if and only if $P_X^a \neq \emptyset$.
\item[(ii)] The set $P_X^a$, if nonempty, is an affine space directed by the vector space $H^0\left(X, K_X^{\otimes 2}(-D)\right)$ of holomorphic quadratic differentials on $X$ vanishing over the divisor $D$.
\end{itemize}
\end{proposition}

Let us state point (i) in a slighty more general framework.
Let $\mathcal S$ be a sheaf of abelian groups on $X$.
By a \emph{sheaf of affine spaces directed by $\mathcal S$} we mean a sheaf $\mathfrak S$ such that if $U \subset X$ is an open subset, the space of sections $\Gamma(U, \mathfrak S)$ is either empty or an affine space directed by $\Gamma(U, \mathcal S)$.
We say moreover that $\mathfrak S$ is \emph{locally nonempty} if for any $x \in X$ there exists an open neighborhood $U$ of $X$ such that $\Gamma(U, \mathfrak S) \neq 0$.

\begin{lemma}\label{lemma_obstruction_affine_sheaf}
Let $\mathcal S$ be a sheaf of abelian groups over the Riemann surface $X$ and $\mathfrak S$ a locally nonempty sheaf of affine spaces directed by $\mathfrak S$. 
There exists an obstruction $c(\mathfrak S) \in H^1(X, \mathcal S)$ such that $c(\mathfrak S)=0$ if and only if $H^0(X, \mathfrak S) \neq 0$.
\end{lemma}

\begin{proof}
Since $\mathfrak S$ is locally nonempty, there exists an open cover $(U_i)_{i \in I}$ of $X$ such that for each $i \in I$ there is a local section $s_i \in \Gamma\left(U_i, \mathfrak S\right)$.
For any $i, j \in I$, write $c_{ij} = s_j-s_i \in \Gamma\left(U_{ij}, \mathcal S\right)$ where $U_{ij} = U_i \cap U_j$.
The family $\left(c_{ij}\right)_{i, j \in I}$ is clearly a $1$-cocycle for $\mathcal S$ and thus defines a cohomology class $c(\mathfrak S) \in H^1(X, \mathcal S)$.

The class $c(\mathfrak S)$ does not depend on the choice of the family $\left(s_i\right)_{i \in I}$.
Indeed, if $\left(s_i'\right)_{i \in I}$ is another such family and $c_{ij}' = s_j'-s_i'$ the corresponding cocycle, then one has $c_{ij}' - c_{ij} = \left(s_j'-s_j\right) - \left(s_i'-s_i\right)$, so that the cohomology class of the cocycle $\left(c_{ij}' - c_{ij}\right)_{i, j \in I}$ is $0$.
In particular, if $H^0(X, \mathfrak S) \neq 0$ and $s$ is a global section of $\mathfrak S$, then one can take $s_i = s|_{U_i}$ in which case $c_{ij} = 0$ thus $c(\mathfrak S) = 0$.

Reciprocally, if $c(\mathfrak S)=0$, then up to refining the open cover $(U_i)_{i \in I}$, there is a family $\left(t_i\right)_{i \in I}$ such that $t_i \in \Gamma\left(U_i, \mathcal S\right)$ and $c_{ij} = t_j-t_i$.
In that case, write $\widetilde s_i = s_i - t_i$.
The $\widetilde s_i$ coincide on the intersections $U_{ij}$ and thus glue together to provide a global section of $\mathfrak S$.
\end{proof}

\begin{proof}(of proposition \ref{proposition_space_bps_given_bc})

Let us first focus on the affine structure of $P_X^a$ in case it is nonempty.
Let $z$ be la local coordinate defined on an open subset $U \subset X$.
Let $x_0 \in U$ and let $a_0 \in A_X^{nx_0}$ be a branching class of order $n$ at $x_0$. 
Let also $\phi_1 = \varphi_1(z), \phi_2 = \varphi_2(z)$ be nonconstant holomorphic functions on $U$ such that $\phi_1, \phi_2$ are both ramified at $x_0$ with $(n+1)$ branches and branching class $a_0$.
Suppose $\phi_1$ and $\phi_2$ have no other ramification point.
Recall that $\left\{\phi_1, z\right\}$ denotes the Schwarzian derivative $\left(\varphi_1''/\varphi_1'\right)' - \frac 1 2 \left(\varphi_1''/\varphi_1'\right)^2$.
It is well-known (see for instance \cite{gunning}) that the difference $\left\{\phi_2, z\right\} - \left\{\phi_1, z\right\}$ does not depend on the choice of coordinate $z$. 
Moreover, this difference vanishes uniformly on $U$ if and only if there is a Möbius transformation $g \in \PSL(2, \C)$ such that $\phi_2 = g \circ \phi_1$.
A straightforward calculation shows the following :
\begin{itemize}
\item[(i)] The quantity $\left\{\phi_1, z\right\}dz^{\otimes 2}$ is a meromorphic quadratic differential on $U$ that is holomorphic away from $x_0$ and admits a pole of order $2$ at $x_0$.
\item[(ii)] The coefficient of order $-2$ at $x_0$ of $\left\{\phi_1, z\right\}dz^{\otimes 2}$ is $\frac{1 - (n+1)^2}2$.
\item[(ii)] The coefficients of order $-1, 0, 1, \dots n-1$ at $x_0$ of $\left\{\phi_1, z\right\}dz^{\otimes 2}$ determine and are determined by the branching class of $\phi_1$ at $x_0$. In particular, the difference $\left\{\phi_2, z\right\}dz^{\otimes 2} - \left\{\phi_1, z\right\}dz^{\otimes 2}$ is a holomorphic quadratic differential on $U$ with a zero of order $n$ at $x_0$.
\end{itemize}
Moreover, it is a consequence of Fuchs's local theory (see for instance \cite{mandelbaum_1} or \cite{saint-gervais_en}, section IX.1) that if $\omega$ is a holomorphic quadratic differential on $U$ with a zero of order $n$ at $x_0$, then on an open neighborhood $V$ of $x_0$ in $U$, there exists a function $\phi_3$ such that $\omega = \left\{\phi_3, z\right\}dz^{\otimes 2} - \left\{\phi_1, z\right\}dz^{\otimes 2}$.
The function $\phi_3$ is ramified at $x_0$, with $(n+1)$ branches and branching class $a_0$.

Now suppose $P^a \neq \emptyset$ and let $p_1, p_2 \in P_X^a$.
Let respectively $\left(U_i, \phi_i\right)$ and $\left(U_i, \psi_i\right)$ be atlases for $p_1$ and $p_2$ and for each $i \in I$ let $z_i$ be a coordinate on $U_i$.
The above facts about schwarzian derivatives imply that the local holomorphic quadratic differentials $\left \{\psi_i, z_i\right\}dz_i^{\otimes 2} - \left \{\phi_i, z_i\right\}dz_i^{\otimes 2}$ coincide on intersections $U_i \cap U_j$ and the global quadratic differential $\omega$ obtained by gluing them is a section of $H^0\left(X, K_X^{\otimes 2}(-D)\right)$.
We write $p_2-p_1 := \omega$. Fuchs's local theory ensures that any quadratic differential in $H^0\left(X, K_X^{\otimes 2}(-D)\right)$ is obtained as the difference of two branched projective structures in $P_X^a$. This shows that when $P_X^a \neq \emptyset$, it is an affine space directed by $H^0\left(X, K_X^{\otimes 2}(-D)\right)$.

Now we do not suppose $P_X^a \neq \emptyset$ anymore.
Denote by $\mathfrak P_X^a$ the \emph{sheaf of branched projective structures} on $X$ of divisor $D$ and branching class $a$ : if $U \subset X$ is an open subset, the space of local sections $\Gamma\left(U, \mathfrak P_X^a\right)$ is the space $P_U^{a|_U}$ of branched projective structures on $U$ of divisor $D|_U$ and branching class $a|_U$.
As we have just proved, the sheaf $\mathfrak P_X^a$ is an affine sheaf directed by $K_X^{\otimes 2}(-D)$.
The sheaf $\mathfrak P_X^a$ is moreover locally nonempty : if $x \in X$, any local holomorphic function around $x$ with the right branching order and branching class at $x$ provides a local section of $\mathfrak P_X^a$.
Thus by lemma \ref{lemma_obstruction_affine_sheaf}, there is a canonical cohomology class $c\left(\mathfrak P_X^a\right) \in H^1\left(X, K_X^{\otimes 2}(-D)\right)$ that vanishes if and only if $H^0\left(X, \mathfrak P_X^a\right) \neq 0$.
We write $c(a) := c\left(\mathfrak P_X^a\right)$.
\end{proof}

\section{Familes of Branched Projective Structures over a Family of Curves}\label{section_families_branched_projective_structures}

By a \emph{family of complex curves}, we mean a triple $(S, \mathcal X, \pi)$ where $\mathcal X$ and $S$ are connected, reduced analytic spaces and $\pi : \mathcal X \to S$ is a proper and smooth morphism.
The analytic space $S$ is called the \emph{base} of $\mathcal F$.
It is equivalent to require that $\pi$ is proper and locally trivial on $\mathcal X$ : each point of $\mathcal X$ admits a neighborhood $U \subset \mathcal X$ such that $V = \pi(U) \subset S$ is open and there exists an isomorphism of analytic spaces $\phi : U \simeq V \times \Omega$ where $\Omega$ is an open subset of $\C$ and $\pi = \pr_1 \circ \phi$, where $\pr_1$ is the projection on the first factor.
In particular, the fibers of $\pi$ are smooth and in the case where $S$ is smooth, $\pi$ is a holomorphic submersion. 
According to a theorem by Ehresmann (see for instance \cite{voisin_en_I}), $\pi : X \to S$ is topologically a locally trivial bundle.

Given a family of complex curves $\mathcal F = (S, \mathcal X, \pi)$, an analytic space $S'$ and a morphism of analytic spaces $f : S' \to S$, we denote by $f^*\mathcal F = \left(S', f^* \mathcal X, f^* \pi\right)$ the \emph{pullback family} of $\mathcal F$ by $\pi$ : $f^*\mathcal X$ is the fiber product $S' \times_S \mathcal X$ and $f^*\pi$ is the projection on the first factor.

\subsection{Families of Divisors over a Family of Curves}\label{subsection_families_divisors}

By a \emph{family of complex curves with divisors}, or more shortly \emph{family of divisors}, we mean a family $\mathcal F = (S, \mathcal X, \mathcal D, \pi)$, where $(S, \mathcal X, \pi)$ is a family of complex curves, and $\mathcal D$ is an effective Cartier divisor on $\mathcal X$ that does not contain any fiber of $\pi$.
Writing $X_s = \pi^{-1}(\{s\})$ ($s \in S$), Rouch\'{e}'s theorem implies that the degree of $D_s = D|_{X_s}$ does not depend on $s$.
The \emph{degree} of the family $\mathcal F$ is the degree of the divisors $D_s$.

As in the case of families of curves, if $f : S' \to S$ is a morphism of analytic spaces and $\mathcal F$ a family of divisors, one can define the \emph{pullback} $f^*\mathcal F$ of $\mathcal F$ by $f$ : the base of $f^*\mathcal F$ is $S$ and the fiber of $f^* \mathcal F$ over $s' \in S'$ is isomorphic to the fiber of $\mathcal F$ over $f(s') \in S'$.

Let $\mathcal G$ be a family of complex curves, with base an analytic space $T$.
A family of divisors \emph{over} $\mathcal G$ is a family of divisors $\mathcal F = (S, \mathcal X, \mathcal D, \pi)$ along with a morphism $f : S \to T$, such that the underlying family of curves $(S, \mathcal X, \pi)$ is the pullback of $\mathcal G$ by $f$ : $(S, \mathcal X, \pi) = f^*\mathcal G$.
If $g : S' \to S$ is a morphism of analytic spaces, the pullback of $\mathcal F$ by $g$ is still a family of divisors over $\mathcal G$, the associated morphism being $f \circ g : S' \to T$.
Take $n \in \N$.
The family $\mathcal F$ is a \emph{universal family of divisors of degree $n$ over $\mathcal G$} if for any analytic space $S'$ and any family $\mathcal F'$ of divisors of degree $n$ over $\mathcal G$ with base $S'$, there exists a unique morphism $g : S' \to S$ such that $\mathcal F' = g^*\mathcal F$.
A universal family of divisors of degree $n$ over $\mathcal G$, if it exists, is clearly unique up to a unique isomorphism.

It is a classical fact that for any family $\mathcal G$ of complex curves, there exists a universal family $\mathcal F$ of divisors of degree $n$ over $\mathcal G$.
Let us recall briefly the construction of $\mathcal F$.
Write $\mathcal G = (T, \mathcal X, \pi)$.
We set $S = \Sym^n_T \mathcal X$ the $n$-th symmetric power of $\mathcal X$ relatively to $T$, that comes with a map $\gamma : S \to T$.
The fiber of $\gamma$ over $t \in T$ is the $n$-th symmetric power of the fiber $X_t$ of $\pi$ over $t$.
Note that $S = \{(t, D)|t \in T, D \text{ effective divisor of degree $n$ on $X_t$}\}$.
The universal family $\mathcal F$ is defined as $(S, \gamma^*\mathcal X, \gamma^*\pi, \mathcal D)$ where $\mathcal D$ is the Cartier divisor on $\gamma^*\mathcal X$ such that for any $s = (t, D) \in S$, $\mathcal D|_{X_s} = D$, where $X_s = \gamma^*\pi^{-1}(\{s\})$.
If $\mathcal F' = (S', \mathcal X', \pi', \mathcal D')$, along with a morphism $f' : S' \to T$ is another family of divisors of degree $n$ over $\mathcal G$, $\mathcal F' = g^*\mathcal F$ where $g : S' \to S$ is the analytic morphism defined for any $s' \in S'$ by $g(s') = \left(f'(s'), D'_{s'}\right)$ where $X'_{s'}$ and $X_{f'(s')}$ are identified (here $X_t = \pi^{-1}(\{t\})$, $X'_{s'} = \pi'^{-1}(\{s'\})$ and $D'_{s'} = \mathcal D'|_{X_{s'}}$).
The morphism $g$ is clearly the only one such that $\mathcal F' = g^*\mathcal F$.
See \cite{barlet-magnusson} for more details on spaces of effective divisors.

\subsection{Families of Branching Classes over a Family of Divisors}\label{subsection_families_branching_classes}

The following definition gives a meaning to the notion of a holomorphic deformation of a branching class.

\begin{definition}\label{definition_branched_atlas}
\begin{itemize}
\item[(i)]
Let $X$ be a Riemann surface and $D$ an effective divisor on $X$.
A \emph{branched atlas} on the complex curve $X$ with branching divisor $D$ is the datum of an open covering $(U_i)_{i \in I}$ of $X$, along with holomorphic functions $f_i : U_i \to \mathbb C$ whose branching divisors are $D|_{U_i}$, such that the branching class of $f_i$ on $U_{ij} = U_i \cap U_j$ is the same as the branching class of $f_j$ on $U_{ij}$.
\item[(ii)]
Let $(S, \mathcal X, \pi, \mathcal D)$ be a family of divisors.
A \emph{branched atlas on $\mathcal X$ relative to $S$, with branching divisor $\mathcal D$} is the datum of an open covering $(U_i)_{i \in I}$ of $\mathcal X$, along with holomorphic functions $f_i : U_i \to \mathbb C$ such that for  any $s \in S$, the family $(U_i \cap X_s, f_i|_{U_i \cap X_s})_{i \in I}$ is a branched atlas on $X_s$ of divisor $D_s$.
Two branched atlases are \emph{equivalent} if their union is still a branched atlas.
A \emph{branching class on $\mathcal X$ relative to $S$, with branching divisor $\mathcal D$} is an equivalence class of branched atlases on $\mathcal X$ relative to $S$, with branching divisor $\mathcal D$.
\end{itemize}
\end{definition}

We are now able to use a vocabulary similar to the one in subsection \ref{subsection_families_divisors}.

\begin{definition}\label{definition_families_bc}
\begin{itemize}
\item[(i)]
Let $n \in \N$.
A \emph{family of complex curves with divisors (of degree $n$) and branching classes}, or more shortly a \emph{family of branching classes (of degree $n$)} is given by a family $\mathcal F = (S, \mathcal X, \pi, \mathcal D, \mathfrak a)$ where $(S, \mathcal X, \pi, \mathcal D)$ is a family of divisors (of degree $n$) and $\mathfrak a$ is a branching class on $\mathcal X$ relative to $S$, with branching divisor $\mathcal D$.
\item[(ii)]
Let $\mathcal G$ be a family of divisors (respectively of complex curves) of basis $T$.
A \emph{family of branching classes over $\mathcal G$} is a family of branching classes $\mathcal F$ of basis $S$, along with a morphism $f : S \to T$ such that the underlying family of divisors (respectively of complex curves) is $f^*\mathcal G$.
\item[(iii)]
Let $S$ be an analytic space, $\mathcal F$ a family of branching classes over $\mathcal G$ with basis $S$.
The family $\mathcal F$ is said to be \emph{universal} if any family $\mathcal F'$ of branching classes over $\mathcal G$ is uniquely obtained as the pullback of $\mathcal F$ by some morphism $g : S' \to S$, where $S'$ is the basis of $\mathcal F'$.
\end{itemize}
\end{definition}

The discussion that will follow in this subsection can be summarized in the following proposition :

\begin{proposition}\label{proposition_universal_family_bc_over_divisor}
Let $\mathcal G = (T, \mathcal X, \pi, \mathcal D)$ be a family of divisors.
There exists a universal family $\mathcal F$ of branching classes over $\mathcal G$.

Moreover, the basis of $\mathcal F$ is an affine bundle over $T$, directed by the vector bundle $\pi_*K_{\mathcal X/T}|_{\mathcal D}$ of relative differentials on $\mathcal X$ restricted to $\mathcal D$ with respect to $T$.
\end{proposition}

\begin{remark}
\begin{itemize}
\item[(i)]
With the notations of proposition \ref{proposition_universal_family_bc_over_divisor}, the fiber over $t \in T$ of the vector bundle $\pi_*K_{\mathcal X/T}|_{\mathcal D}$ is $H^0\left(X_t, K_{X_t}|_{D_t}\right)$
\item[(ii)]
It is clear from definition \ref{definition_families_bc} that the universal family in  \ref{proposition_universal_family_bc_over_divisor} is unique up to a unique isomorphism.
\item[(iii)]
It follows from proposition \ref{proposition_universal_family_bc_over_divisor} that for any family of curves $\mathcal G$, there exists a universal family of branching classes over $\mathcal G$.
This is shown by applying proposition \ref{proposition_universal_family_bc_over_divisor} to the universal family of divisors over $\mathcal G$.
\end{itemize}
\end{remark}

For the rest of the subsection, fix a family of divisors $\mathcal G = (T, \mathcal X, \pi, \mathcal D)$.
Let us denote by 
\begin{equation}\label{eq_base_univ_family_bc_over_divisor}
\mathcal A_{\mathcal X}^{\mathcal D} = \left \{(t, a)|t \in T, a \in A_{X_t}^{D_t} \right \}
\end{equation}
the union of all the spaces of branching classes associated to one of the curves with divisor of the family $\mathcal G$.
There is an obvious projection
\begin{equation}\label{eq_proj_bc_divisors}
\delta : \mathcal A_{\mathcal X}^{\mathcal D} \twoheadrightarrow T
\end{equation}
For each $t \in T$, according to proposition \ref{proposition_affine_structure_branching_classes}, the fiber $\delta^{-1}(s) = A_{X_s}^{D_s}$ is an affine space directed by the vector space $H^0(X_t, K_{X_t}|_{D_t}) = \left(\pi_*\left( K_{\mathcal X/T}|_{\mathcal D} \right) \right)_t$.
In order to get on $\mathcal A_{\mathcal X}^{\mathcal D}$ the structure of an affine bundle on $T$ directed by the vector bundle $\pi_*\left( K_{\mathcal X/S}|_{\mathcal D} \right)$, it is enough to specify local holomorphic sections of $\delta$.

Remark that, up to restricting $T$, there exists a branched atlas on $\mathcal X$ relative to $T$ with branching divisor $\mathcal D$.
Indeed, take $t \in T$ and to any point $y_k$ of $D_t = \sum_{k = 1}^r n_k \cdot y_k$, associate a neighborhood $U_k$ of $y_k$ in $\mathcal X$ and a function $g_k : U_k \to \mathbb C$ such that $\mathcal D \cap U_k$ is the zero locus of $g_k$.
Write $V = \pi(U_1) \cap \dots \cap \pi(U_r) \subset T$.
It is a neighborhood of $t$.
Let also $(W_i, f_i)_{i \in I}$ be a relative atlas (i.e. a branched relative atlas with divisor zero) for the family $\pi^{-1}(V) \backslash \mathcal D$.
Such a relative atlas exists because the family $(T, \mathcal X, \pi)$ is locally trivial on $\mathcal X$.
The union of the families $(W_i, f_i)_{i \in I}$ and $(U_k, g_k)_{1 \le k \le r}$ is a branched relative atlas for the family $\pi^{-1}(V)$ with branching divisor $\mathcal D|_{\pi^{-1}(V)}$.

The datum of a branching class $\mathfrak a$ on $\mathcal X$ relative to $T$ with branching divisor $\mathcal D$ defines a section $\sigma_{\mathfrak a} : T \to \mathcal A_{\mathcal X}^{\mathcal D}$ of the projection $\delta$ : $\sigma_{\mathfrak a}(t)$ is the branching class of the branched atlas $\mathfrak a$ restricted to $X_t$.

\begin{lemma}\label{lemma_diff_atlas_proj}
Let $\mathfrak a_1$ and $\mathfrak a_2$ be two branching classes on $\mathcal X$ relative to $T$ with divisor $\mathcal D$.
For each $t \in T$, write $\sigma_{\mathfrak a_2}(t) = \sigma_{\mathfrak a_1}(t) + \phi(t) \in A_{X_t}^{D_t}$, with $\phi(t) \in H^0(X_t, K_{X_t}|_{D_t})$.
Then $\phi$ is a holomorphic section of the vector bundle $\pi_*(K_{\mathcal X/T}|_{\mathcal D})$.
\end{lemma}

\begin{proof}
Take $t_0 \in T$, and let $y \in X_t$ be a point of $D_t$.
The local triviality on $\mathcal X$ of $(T, \mathcal X, \pi)$ implies that there exists a local neighborhood $V$ of $t_0$ in $T$, an open neighborhood $U \subset \mathcal X$ of $y$ with $\pi(U) = V$, and local coordinates $(z, w): U \to V  \times \mathbb C$ with $\pi(z, w) = z$.
Without loss of generality, suppose $U$ is the domain of a chart both in $\mathfrak a_1$ and $\mathfrak a_2$.
Take charts $f, g : U \to \mathbb C$ of $\mathfrak a_1$ and $\mathfrak a_2$ respectively.
According to the proof of proposition \ref{proposition_affine_structure_branching_classes}, it is enough to show that the function
$$\frac{\partial^2_w g}{\partial_w g} - \frac{\partial^2_w f}{\partial_w f} : U \to \mathbb C$$
is holomorphic.

The functions $f$ and $g$ have the same branching divisor when restricted to each fiber of $X$, i.e. the functions $\partial_w g$ et $\partial_w f$ have the same vanishing locus.
Thus there exists a nonvanishing holomorphic function $h : U \to \mathbb C^*$ such that $\partial_w g = h \cdot \partial_w f$.
Thus one has $\partial^2_w g /\partial_w g - \partial^2_w f / \partial_w f = \partial_w h / h$, which is holomorphic.
\end{proof}

As a consequence of lemma \ref{lemma_diff_atlas_proj}, $\mathcal A_{\mathcal X}^{\mathcal D}$ admits a structure of affine bundle for which the sections coming from branched relative atlases on $\mathcal G$ are holomorphic.

\begin{lemma}\label{lemma_branched_atlas_holomorphic_section}
Let $\sigma : T \to \mathcal A_{\mathcal X}^{\mathcal D}$ be a holomorphic section. 
There exists a branching class $\mathfrak a$ on $\mathcal X$, relative to $T$ and of divisor $\mathcal D$ such that $\sigma = \sigma_{\mathfrak a}$.
\end{lemma}

\begin{proof}
The union of two branched relative atlases with the same branching classes given by $\sigma$ is still a branched relative atlas with relative branching classes given by $\sigma$.
As a consequence, it is enough to prove the lemma locally on $T$.

Take $t \in T$.
Up to restricting $T$ to an open neighborhood of $t$, there exists a relative branching class $\mathfrak a_0$ on $\mathcal G$, with associated section $\sigma_{\mathfrak a_0} : T \to \mathcal A_{\mathcal X}^{\mathcal D}$.
Let $y \in X_t$ be a point of $D_t$.
Let $(U, f)$ be a chart of an atlas in $\mathfrak a_0$ with $y \in U$.
Since $(T, \mathcal X, \pi)$ is locally trivial on $\mathcal X$, up to restricting $U$, one has a local biholomorphism $(z, w) : U \to V \times \C$, where $V = \pi(U)$ and $\pi(z, w) = z$.
The difference $\sigma-\sigma_{\mathfrak a_0}$ is a section of the vector bundle $\pi_*(K_{\mathcal X/T}|_{\mathcal D})$ on $T$, that can be seen as a section of the vector bundle $K_{\mathcal X / T}|_{\mathcal D}$ on the analytic space $\mathcal D$. 
With this point of view, up to restricting $U$, $(\sigma-\sigma_{\mathfrak a_0})|_U$ is given by $\phi|_{\mathcal D}$, where $\phi$ is a section on $U$ of $K_{\mathcal X/T}$, given by $\phi = \varphi(z, w)dw$.
Up to restricting $U$ again (in particular so that it is contractible), there exists a holomorphic function that never vanishes $h : U \to \mathbb C^*$ such that $\varphi = \partial_w h / h$. By restricting $U$ once more, one gets a function $g : U \to \mathbb C$ such that $\partial_w g = h \cdot \partial_w f$.
The branching divisor of $g$ is then $\mathcal D|_U$ and its branching classes on the fibers of $\pi$ are given by $\sigma$.

By doing so at each point of $D_t$, by restricting $T$ to the intersection of the obtained open sets $\pi(U)$, and then by completing with a relative atlas on $\mathcal X \backslash \mathcal D$, one gets a branched relative atlas $\mathfrak a$ on $\mathcal G$ with $\sigma = \sigma_{\mathfrak a}$.
\end{proof}

The pullback family of divisors $\delta^*\mathcal G = \left(\mathcal A_{\mathcal X}^{\mathcal D}, \delta^*\mathcal X, \delta^* \pi, \delta^* \mathcal D\right)$ comes with a tautological relative branching class $\mathfrak a_0$~: if $(t, a) \in \mathcal A_{\mathcal X}^{\mathcal D}$, the branching class defined by $\mathfrak a_0$ on $\left(\delta^*\pi\right)^{-1}(t, a) = X_t$ with divisor $D_t$ is $a$.
Let us show that $\mathfrak a_0$ is indeed a relative branching class.
The pullback by $\delta$ of a relative branching class on $\mathcal X$ relative to $T$ of divisor $\mathcal D$ is clearly a branching class on $\delta^* \mathcal X$, relative to $\mathcal A_{\mathcal X}^{\mathcal D}$ of divisor $\delta^* \mathcal D$.
Thus the analytic space $\mathcal A_{\delta^* \mathcal X}^{\delta^*\mathcal D}$ of branching classes on the fibers of $\delta^* \mathcal G$ is the analytic fiber product $\mathcal A_{\mathcal X}^{\mathcal D} \times_T \mathcal A_{\mathcal X}^{\mathcal D}$.
The tautological relative branching class $\mathfrak a_0$ comes from the diagonal section of $\mathcal A_{\delta^* \mathcal X}^{\delta^*\mathcal D}$, that is holomorphic.

The family $\mathcal F := \left(\mathcal A_{\mathcal X}^{\mathcal D}, \delta^*\mathcal X, \delta^* \pi, \delta^* \mathcal D, \mathfrak a_0\right)$, along with the morphism $\delta : \mathcal A_{\mathcal X}^{\mathcal D} \to T$ is a family of branching classes over $\mathcal G$.

\begin{lemma}\label{lemma_univ_family_bc}
The family $\mathcal F$ above, along with the map $\delta$, is a universal family of branching classes over $\mathcal G$.
\end{lemma}

\begin{proof}
Let $\mathcal F' = (S', \mathcal X', \pi', \mathcal D', \mathfrak a')$, along with a morphism $f' : S' \to T$, be a family of branching classes over $\mathcal G$. For any $s' \in S'$, the relative branched atlas $\mathfrak a'$ defines a branching class $a_{s'}$ over $\left(X'_{s'}, D'_{s'}\right) \simeq \left(X_{f'(s')}, D_{f'(s')}\right)$.
Thus we get a map $g : S' \to \mathcal A_{\mathcal X}^{\mathcal D}$ given by $g(s') = (f'(s'), a_{s'})$.
One has $f' = \delta \circ g$, for any $s' \in S'$, the fiber of $\mathcal F'$ over $s'$ is identified to the fiber of $\mathcal F$ over $g(s')$, and $g$ is the only map from $S'$ to $\mathcal A_{\mathcal X}^{\mathcal D}$ that fulfills these two conditions.

To complete the proof of the lemma, we only need to show that $g$ is analytic.
Clearly, $\mathcal A_{\mathcal X'}^{\mathcal D'}$ is the analytic fiber product $\mathcal A_{\mathcal X'}^{\mathcal D'} = S' \times_T \mathcal A_{\mathcal X}^{\mathcal D}$.
By lemma \ref{lemma_branched_atlas_holomorphic_section}, the relative branching class $\mathfrak a'$ gives an analytic section $\sigma_{\mathfrak a}$ of $\mathcal A_{\mathcal X'}^{\mathcal D'}$, that can be viewed as an analytic morphism $\sigma_{\mathfrak a} : S' \to S' \times_T \mathcal A_{\mathcal X}^{\mathcal D}$.
The map $g$ is this analytic morphism composed with the projection on $A_{\mathcal X}^{\mathcal D}$, which is analytic.
\end{proof}

\subsection{Families of Branched Projective Structures}\label{subsection_families_bps}

The following definition gives a meaning to the notion of a holomorphic deformation of a branched projective structure.

\begin{definition}\label{definition_relative_bps}
Let $(S, \mathcal X, \pi, \mathcal D, \mathfrak a)$ be a family of branching classes.
A \emph{branched projective atlas on $\mathcal X$ relative to $S$ with branching divisor $\mathcal D$ and relative branching class $\mathfrak a$} is the datum of an atlas on $\mathcal X$, relative to $S$, of divisor $D$ and relative branching class $\mathfrak a$ such that the branched atlas induced on each $X_s$ is a branched projective atlas.
Two relative branched projective atlases are \emph{equivalent} if their union is still a relative branched projective atlas.
A \emph{branched projective structure} relative to $S$ with branching divisor $\mathcal D$ and relative branching class $\mathfrak a$ is an equivalence class of branched projective atlases on $\mathcal X$ relative to $S$ with branching divisor $\mathcal D$ and relative branching class $\mathfrak a$.
\end{definition}

We are now able to introduce a vocabulary similar to the ones in subsections \ref{subsection_families_divisors} and \ref{subsection_families_branching_classes}.

\begin{definition}\label{definition_families_bps}
\begin{itemize}
\item[(i)]
Let $n \in \N$.
A \emph{family of complex curves with divisors (of degree $n$), branching classes and branched projective structures}, or more shortly a \emph{family of branched projective structures (of degree $n$)} is given by a family $\mathcal F = (S, \mathcal X, \pi, \mathcal D, \mathfrak a, \mathfrak p)$ where $(S, \mathcal X, \pi, \mathcal D, \mathfrak a)$ is a family of branching classes (of degree $n$) and $\mathfrak p$ is a branched projective structure on $\mathcal X$ relative to $S$, with branching divisor $\mathcal D$ and branching class $\mathfrak a$.
\item[(ii)]
Let $\mathcal G$ be a family of branching classes (respectively of complex curves, respectively of divisors) of basis $T$.
A \emph{family of branched projective structures over $\mathcal G$} is a family of branched projective structures $\mathcal F$ of basis $S$, along with a morphism $f : S \to T$ such that the underlying family of branching classes (respectively of complex curves, respectively of divisors) is $f^*\mathcal G$.
\item[(iii)]
Let $S$ be an analytic space, $\mathcal F$ a family of branched projective structures over $\mathcal G$ with basis $S$.
The family $\mathcal F$ is said to be \emph{universal} if any family $\mathcal F'$ of branched projective structures over $\mathcal G$ is uniquely obtained as the pullback of $\mathcal F$ by some morphism $g : S' \to S$, where $S'$ is the basis of $\mathcal F'$.
\end{itemize}
\end{definition}

The aim of this section is to prove the following proposition :

\begin{proposition}\label{proposition_universal_family_bps_over_bc}
Let $\mathcal G$ be a family of branching classes. There exists a universal family of branched projective structures over $\mathcal G$.
\end{proposition}

\begin{remark}
It is clear from definition \ref{definition_families_bc} that the universal family in proposition \ref{proposition_universal_family_bps_over_bc} is unique up to a unique isomorphism.
\end{remark}

Before diving into the proof of proposition \ref{proposition_universal_family_bps_over_bc} in the next subsections, let us state a corollary that is one of the main results of this paper.
Let us firsrt recall the definition of a \emph{marked} Riemann surface.

\begin{definition}
Let $X$ be a compact Riemann surface.
Let $\mathcal F = (X, S, \pi)$ be a family of Riemann surfaces
\begin{itemize}
\item[(i)] If $X$ has genus $0$, a \emph{marking} on $X$ is the datum of $3$ distinct points on $X$.
A marking on $\mathcal F$ is the datum of three holomorphic sections of $\pi$ that do not intersect with each other.
\item[(ii)] If $X$ has genus $1$ and $S$ is a differential surface of genus $1$, a \emph{marking on $X$ with reference $S$} is the datum of a point in $X$ and an isotopy class of diffeomorphisms from $S$ to $X$.
A marking on $\mathcal F$ with reference $S$ is the datum of a holomorphic section of $\pi$, and an isotopy class of diffeomorphisms from $S \times X$ to $\mathcal X$ that commute with $\pi$.
\item[(iii)] If $X$ has genus $g \ge 2$ and $S$ is a differential surface of genus $g$, a \emph{marking on $X$ with reference $S$} is the datum of an isotopy class of diffeomorphisms from $S$ to $X$.
A marking on $\mathcal F$ with reference $S$ is the datum of an isotopy class of diffeomorphisms from $S \times X$ to $\mathcal X$ that commute with $\pi$.
\end{itemize}
A \emph{marked compact Riemann surface} is a compact Riemann surface with a marking, where a differential surface $S$ is supposed fixed.
\end{definition}

In particular, there is at most one isomorphism between two marked Riemann surfaces.

Marked (families of) divisors (respectively branching classes, respectively branched projective structures) are (families of) divisors (respectively branching classes, respectively branched projective structures) along with a marking on the underlying (family of) curves.
A  marked family of divisors (respectively branching classes, respectively branched projective structures) $\mathcal F$ is said to be \emph{over} a marked family of curves (respectively divisors, respectively branching classes) $\mathcal G$ if $\mathcal F$ is given as the pullback of $\mathcal G$ by an analytic morphism.

Propositions \ref{proposition_universal_family_bc_over_divisor} and \ref{proposition_universal_family_bps_over_bc} remain true if we consider marked families : the marking on the universal family is obtained by pulling back the marking of the base family.

Let us give a definition of a \emph{universal family}, that is not over another family :

\begin{definition}\label{definition_universal_family}
Take $n, g \in \N$
A family $\mathcal F$ of (marked) Riemann surfaces (respectively divisors of degree $n$, respectively branching classes of degree $n$, respectively branched projective structures of degree $n$) of genus $g$ is said to be \emph{universal} if any other such family $\mathcal F'$ can be given in a unique way as a pullback of $\mathcal F'$.
\end{definition}

Let $g \in \N$. There is no universal family of curves of genus $g$.
In particular the tautological family over the moduli space $\mathcal M_g$ of genus $g$ curves is not universal, since nontrivial isotrivial families are not pullbacks of it.
There exists however a universal family of \emph{marked} Riemann surfaces of genus $g$, given by $\mathcal U_g = \left(\mathcal T_g, \mathcal X_g, \pi_g, m_g\right)$, where $\mathcal T_g$ is the Teichmüller space for curves of genus $g$, $\mathcal X_g$ is the tautological family over $\mathcal T_g$, $\pi_g$ is the projection and $m_g$ is the tautological marking.
For $n \in \N$, denote by $\mathcal U_g^{\Div}(n)$ the marked family of divisors of degree $n$ obtained as the universal family of divisors of degree $n$ over $\mathcal U_g$, along with the pullback marking.

Let $\mathcal U_g^{\BC}(n)$ be the universal marked family of branching classes over $\mathcal U_g^{\Div}(n)$, given by proposition \ref{proposition_universal_family_bc_over_divisor}, and let $\mathcal U_g^{\BPS}(n)$ be the universal marked family of branched projective structures over $\mathcal U_g^{\BC}(n)$, given by proposition \ref{proposition_universal_family_bps_over_bc}.
The family $\mathcal U_g^{\BPS}(n)$ is clearly the universal family of marked branched projective structures of genus $g$ and degree $n$.
Denote by $\mathcal P_g(n)$ the basis of the family $\mathcal U_g^{\BPS}(n)$.
Now let $p$ be a marked branched projective structure of degree $n$ and genus $n$.
The structure $p$ can be seen as a family over a single-point analytic space $\{x\}$.
By universality of $\mathcal U_g^{\BPS}(n)$, the exists a unique morphism $f : \{x\} \to \mathcal P_g(n)$ such that the family $p$ is the pullback of $\mathcal U_g^{\BPS}(n)$ by $f$.
To the structure $p$ is thus associated a unique point in $\mathcal P_g(n)$.
In particular, the structure $p$ is isomorphic to the fiber of $\mathcal U_g^{\BPS}(n)$ over the associated point.
Clearly, two marked branched projective structures of genus $g$ and degree $n$ have the same associate point if and only if they are isomorphic, and any point in $\mathcal P_g(n)$ is associated to the fiber of $\mathcal U_g^{\BPS}(n)$ over itself.
In this way we get a bijection between the analytic space $\mathcal P_g(n)$ and the set of isomorphism classes of marked branched projective structures of degree $n$ and genus $g$.
We thus have the following corollary of proposition \ref{proposition_universal_family_bps_over_bc} :

\begin{theorem}\label{theorem_moduli_space_bps}
Let $n, g \in \N$. The moduli space of isomorphism classes of marked branched projective structures of genus $g$ and degree $n$ is an analytic space $\mathcal P_g(n)$, and is the basis of a universal family $\mathcal U_g^{\BPS}(n)$ of marked branched projective structures of genus $g$ and degree $n$.
Moreover, the fiber of $\mathcal U_g^{\BPS}(n)$ over a branched projective structure $p \in \mathcal P_g(n)$ is $p$ itself.
\end{theorem}

Let us now give a proof for proposition \ref{proposition_universal_family_bps_over_bc}

\subsection{Schwarzian Parametrization in Family}\label{subsection_schwarzian_parametrization_in_family}

Let $g, n \in \N$.
Fix a family of branching classes $\mathcal G = (T, \mathcal X, \pi, \mathcal D, \mathfrak a)$ of genus $g$ and degree $n$.

\begin{lemma}\label{lemma_relative_bps_schwarzian_derivative}
Suppose given $\mathfrak p_0$ a branched projective structure on $\mathcal X$ relative to $T$ and of relative branching class $\mathfrak a$.
\begin{itemize}
\item[(i)] If $\mathfrak p$ is a branched projective structure on $\mathcal X$ relative to $T$ and of relative branching class $\mathfrak a$, then the difference $\mathfrak p-\mathfrak p_0$ is a holomorphic section of the vector bundle $K_{\mathcal X / T}^{\otimes 2}(-\mathcal D)$ over $\mathcal X$ : $\mathfrak p-\mathfrak p_0 \in H^0\left(\mathcal X, K_{\mathcal X/T}^{\otimes 2}(-\mathcal D)\right)$.
\item[(ii)] If $q \in H^0\left(\mathcal X, K_{\mathcal X/S}^{\otimes 2}(-\mathcal D)\right)$, then there exists a projective structure $\mathfrak p$ on $\mathcal X$ relative to $T$ and of relative branching class $\mathfrak a$ such that $\mathfrak p - \mathfrak p_0 = q$.
\end{itemize}
\end{lemma}

\begin{proof}
\begin{itemize}
\item[(i)] Let $(z, w)$ be a local coordinate on $U \subset \mathcal X$ adapted to the locally trivial family on $\mathcal X$, $(T, \mathcal X, \pi)$, i.e. with $z = \widetilde z \circ \pi$, where $\widetilde z$ is a local coordinate on $T$. Let $\phi_0 : U \to \CP^1$ be a local chart of $\mathfrak p_0$ on $U$ and $\phi : U \to \CP^1$ a local chart of $\mathfrak p$ on $U$.
The difference $\mathfrak p - \mathfrak p_0$ is the meromorphic relative quadratic differential 
$$q = \left(\left(\partial_w(\partial^2_w \phi / \partial_w \phi) - 1/2 (\partial^2_w \phi / \partial_w \phi)^2\right) - \left(\partial_w(\partial^2_w \phi_0 / \partial_w \phi_0) - 1/2 (\partial^2_w \phi_0 / \partial_w \phi_0)^2\right)\right)dw^{\otimes 2}$$
But $\mathfrak p_0$ and $\mathfrak p$, when restricted to the fibers of $\pi$, define branched projective structures with the same branching class, thus for any $t \in T$, $q|_{X_t}$ is a local holomorphic section of $K_{X_t}^{\otimes 2}(-D_t)$. As a consequence, the relative quadratic differential $q$ is a local holomorphic section of $K_{\mathcal X / T}^{\otimes 2}(-\mathcal D)$.
\item[(ii)] Take $x \in \mathcal X$, $U$ an open neighborhood of $x$, $(U, \phi_0)$ a chart of $\mathfrak p_0$, and $(z, w)$ a coordinate on $U$ adapted to the locally trivial family $(T, \mathcal X, \pi)$.
Write $q|_U = udw^{\otimes 2}$. Up to restricting $U$, according to Cauchy-Lipschitz's theorem, there exists $\phi : U \to \CP^1$ a holomorphic solution of the partial differential equation $\partial_w(\partial^2_w \phi / \partial_w \phi) - 1/2 (\partial^2_w \phi / \partial_w \phi)^2 = u + \left(\partial_w(\partial^2_w \phi_0 / \partial_w \phi_0) - 1/2 (\partial^2_w \phi_0 / \partial_w \phi_0)^2\right)$. As a consequence of proposition \ref{proposition_space_bps_given_bc}, $\phi$ defines on $U$ a relative branched projective structure with relative branching class $\mathfrak a$.
By doing so around each point of $\mathcal X$, one gets a branched projective atlas on $\mathcal X$ relative to $S$ of relative branching class $\mathfrak a$.
This atlas defines a relative branched projective structure $\mathfrak p$, and one has $\mathfrak p-\mathfrak p_0 = q$.
\end{itemize}
\end{proof}

For any $t \in T$, write $a_t$ the branching class defined by $\mathfrak a$ on the fiber of $\mathcal G$ over $t$.
Consider
\begin{equation}\label{eq_base_univ_family_bps_over_bc}
\mathcal P_{\mathcal X}^{\mathfrak a} = \left \{(t, p) | t \in T, p \in P_{X_t}^{a_t}\right \}
\end{equation}
the union of the spaces of branched projective structures on the fibers of $\mathcal G$.
There is an obvious map
\begin{equation}\label{eq_proj_bps_bc}
\alpha : \mathcal P_{\mathcal X}^{\mathfrak a} \to T
\end{equation}

According to proposition \ref{proposition_space_bps_given_bc}, for any $t \in T$, the fiber $\alpha^{-1}(t)$ is either empty or an affine space directed by $H^0\left(X_t, K_{X_t}^{\otimes 2}(-D_t)\right)$, according to whether the cohomology class $c(a_t) \in H^1\left(X_t, K_{X_t}^{\otimes 2}(-D_t)\right)$ vanishes or not.
Thus studying the structure of $\mathcal P_{\mathcal X}^{\mathfrak a}$ is essentially studying the variation with $t$ of the space $H^0\left(X_t, K_{X_t}^{\otimes 2}(-D_t)\right)$ and the cohomology class $c(a_t)$

Write $\mathfrak{BR}_{\mathcal X}^{\mathfrak a}$ the sheaf on $\mathcal X$ of branched projective structures on $\mathcal X$ relative to $T$ of relative branching class $\mathfrak a$ : the space of sections $\Gamma\left(U, \mathfrak{BR}_{\mathcal X}^{\mathfrak a}\right)$ is the set of branched projective structures on $U$ relative to $\pi(U)$ of relative branching class $\mathfrak a|_{\pi(U)}$.
Lemma \ref{lemma_relative_bps_schwarzian_derivative} implies that $\mathfrak{BR}_{\mathcal X}^{\mathfrak a}$ is a sheaf of affine spaces directed by the line bundle $K_{\mathcal X/T}^{\otimes 2}(-\mathcal D)$.
Moreover $\mathfrak{BR}_{\mathcal X}^{\mathfrak a}$ is clearly locally nonempty.
In particular, it comes with a cohomology class $\mathcal C(\mathfrak a) \in H^1\left(\mathcal X, K_{\mathcal X/T}^{\otimes 2}(-\mathcal D)\right)$ with $\mathcal C(\mathfrak a) = 0$ if and only if $H^0(\mathcal X, \mathfrak{BR}_{\mathcal X}^{\mathfrak a}) \neq \emptyset$, i.e. if and only if there exists a branched projective structure on $\mathcal X$, relative to $T$ of relative branching class $\mathfrak a$.

Leray spectral sequence gives an exact sequence (see \cite{voisin}) :
\begin{equation}\label{eq_leray_exact_sequence}
H^1\left(T, \pi_*K_{\mathcal X/T}^{\otimes 2}(-\mathcal D)\right) \to H^1\left(\mathcal X, K_{\mathcal X/T}^{\otimes 2}(-\mathcal D)\right) \xrightarrow{\chi_{\mathcal G}} H^0\left(T, R^1\pi_*K_{\mathcal X/T}^{\otimes 2}(-\mathcal D)\right)
\end{equation}
Write 
\begin{equation}\label{eq_section_obstruction}
\mathfrak c(\mathfrak a) = \chi_{\mathcal G}(\mathcal C(\mathfrak a))
\end{equation}

Recall Riemann-Roch formula for $K_{X_t}^{\otimes 2}(-D_t)$ :
\begin{equation}
h^0\left(X_t, K_{X_t}^{\otimes 2}(-D_t)\right) - h^1\left(X_t, K_{X_t}^{\otimes 2}(-D_t)\right) = 3g-3-n
\end{equation}

Let us finish this subsection by showing proposition \ref{proposition_universal_family_bps_over_bc} in the special cases where the branching degree $n$ is either small enough or large enough when compared to the genus $g$.

Suppose $h^0\left(X, K_{X_t}^{\otimes 2}(-D_t)\right)$, thus $h^1\left(X, K_{X_t}^{\otimes 2}(-D_t)\right)$, does not depend on $t \in T$.
Then the sheaves $\pi_*K_{\mathcal X/T}^{\otimes 2}(-\mathcal D)$ and $R^1\pi_*K_{\mathcal X/T}^{\otimes 2}(-\mathcal D)$ are vector bundles, and the section $\mathfrak c(\mathfrak a)$ in \eqref{eq_section_obstruction} satisfies for any $t \in T$ : $\mathfrak c(\mathfrak a)(s) = c\left(a_t\right)$.

Denote by $Z \subset T$ the zero locus of the section $\mathfrak c(\mathfrak a)$, which is an analytic subspace.
Let $\iota : Z \to T$ be the holomorphic immersion.
In other words, there exists a branched projective structure of divisor $D_t$ and class $a_t$ on $X_t$ if and only if $t \in Z$.
Let $\mathcal G' = \left(Z, \iota^* \mathcal X, \iota^* \pi, \iota^* \mathcal D, \iota^* \mathfrak a\right)$.
The space $Z$ is exactly the image of the map $\alpha$ in \eqref{eq_proj_bps_bc} and the space $\mathcal P_{\mathcal X}^{\mathfrak a}$ in \eqref{eq_base_univ_family_bps_over_bc} satisfies
\begin{equation}
\mathcal P_{\mathcal X}^{\mathfrak a} = \left \{(s, p)|s \in Z, p \in P_{X_t}^{a_t}\right \} = \mathcal P_{\iota^* \mathcal X}^{\iota^* \mathfrak a}
\end{equation}

The map $\chi_{\mathcal G'}$, defined as in \eqref{eq_leray_exact_sequence}, is the zero map.
Moreover, if $Z$ is Stein, the exact sequence \eqref{eq_leray_exact_sequence} applied to the family $\mathcal G'$ has vanishing first term, which implies that the second term also vanishes.
As a consequence, up to restricting $Z$ to an open subset, there are branched projective structures on $\iota^* \mathcal X$ relative to $Z$ and of branching class $\iota^* \mathfrak a$ (such structures do not necessarily exist globally on $Z$, see \cite{zhao}).

By lemma \ref{lemma_relative_bps_schwarzian_derivative}, the map $\alpha : \mathcal P_{\mathcal X}^{\mathfrak a} \twoheadrightarrow Z$ is an affine bundle directed by the vector bundle\break $\left( \iota^* \pi\right)_* K_{\iota^*\mathcal X/Z_{\mathcal X}^{\iota^* \mathfrak a}}^{\otimes 2}\left(-\iota^*\mathcal D\right) = \left. \pi_* K_{\mathcal X/T}^{\otimes 2}(-\mathcal D) \right |_Z$, and whose local holomorphic sections are given by the (local on $Z$) branched projective structures on $\iota^* \mathcal X$ relative to $Z$ and with branching class $\iota^* \mathfrak a$.
In particular, the map $\alpha$ is holomorphic.

The pullback family of branching classes $\alpha^*\mathcal G = \left(\mathcal P_{\mathcal X}^{\mathfrak a}, \alpha^*\mathcal X, \alpha^* \pi, \alpha^* \mathcal D, \alpha^* \mathfrak a\right)$ comes with a tautological relative branched projective structure $\mathfrak p_0$~: if $(t, p) \in \mathcal P_{\mathcal X}^{\mathfrak a}$, the branched projective structure defined by $\mathfrak p_0$ on $\left(\alpha^*\pi\right)^{-1}(t, p) = X_t$ with branching class $a_t$ is $p$.

The discussion around lemma \ref{lemma_univ_family_bc} applies \emph{mutatis mutandis}, to show that the family\break $\left(\mathcal P_{\mathcal X}^{\mathfrak a}, \alpha^*\mathcal X, \alpha^* \pi, \alpha^* \mathcal D, \alpha^* \mathfrak a, \mathfrak p_0\right)$, along with the morphism $\alpha : \mathcal P_{\mathcal X}^{\mathfrak a} \to T$ is a universal family of branched projective structures over $\mathcal G$.

If $n < 2g-2$, by Serre duality, for any $t \in T$, $h^1\left(X_t, K_{X_t}^{\otimes 2}\left(-D_t\right)\right) = h^0\left(X_t, T_{X_t}\left(D_t\right)\right) = 0$, because $\deg\left(T_{X_t}\left(D_t\right)\right) < 0$.
In particular $h^1\left(X_t, K_{X_t}^{\otimes 2}\left(-D_t\right)\right)$ and $h^0\left(X_t, K_{X_t}^{\otimes 2}\left(-D_t\right)\right)$ do not depend on $t \in T$.
In that case the section $c(\mathfrak a)$ takes values in the zero vector bundle, thus $Z = T$ and $\mathcal P_{\mathcal X}^{\mathfrak a}$ is an affine bundle over $T$.

On the other extreme, if $n > 4g-4$, then $\deg\left(K_{X_t}^{\otimes 2}\left(-D_t\right)\right) < 0$, so $h^0\left(X_t, K_{X_t}^{\otimes 2}\left(-D_t\right)\right) = 0$.
Again $h^0\left(X_t, K_{X_t}^{\otimes 2}\left(-D_t\right)\right)$ and $h^1\left(X_t, K_{X_t}^{\otimes 2}\left(-D_t\right)\right)$ are both independant of $t \in T$.
Moreover the fibers of $\alpha : \mathcal P_{\mathcal X}^{\mathfrak a} \to T$ contain at most one point.
In that case $\mathcal P_{\mathcal X}^{\mathfrak a}$ is an analytic subspace of $T$.

One has the following proposition, which is a particular case of proposition \ref{proposition_universal_family_bps_over_bc} :

\begin{proposition}\label{proposition_univ_family_bps_low_high_deg}
Recall $n, g \in \N$, $\mathcal G = (T, \mathcal X, \mathcal D, \pi, \mathfrak a)$ is a family of complex curves of genus $g$ with divisors of degree $n$ and branching classes and $\mathcal P_{\mathcal X}^{\mathfrak a}$ is the set of branched projective structures on the fibers of $\mathcal G$.
\begin{itemize}
\item[(i)] If $n < 2g-2$, the space $\mathcal P_{\mathcal X}^{\mathfrak a}$ defined in \eqref{eq_base_univ_family_bps_over_bc} is an affine bundle on $T$ directed by the vector bundle $\pi_*\left(K_{\mathcal X/T}^{\otimes 2}(-\mathcal D)\right)$.
\item[(ii)] If $n > 4g-4$, the space $\mathcal P_{\mathcal X}^{\mathfrak a}$ is an analytic subspace of $T$.
\end{itemize}
In both cases, $\mathcal P_{\mathcal X}^{\mathfrak a}$ is the basis of a universal family of branched projective structures over $\mathcal G$.
\end{proposition}

\begin{remark}
In the limit case $n = 2g-2$, the set $T_0$ of points $t \in T$ for which $h^1(X_t, K_{X_t}^{\otimes 2}) \neq 0$ is the set of points for which $D_t$ is a canonical divisor of $X_t$. If $T' = T \backslash T_0$, $\alpha^{-1}(T') \subset \mathcal P_{\mathcal X}^{\mathfrak a}|_{T}$ is an affine bundle over $T'$.
It is striking that according to \cite{francaviglia-ruffoni}, the branching divisor being canonical for the underlying complex curve is necessary for a branched projective structure to be a critical point of the holomorphic map $\mathcal M_{\lambda, \rho} \to \mathcal T_g$. Here $\mathcal T_g$ is the Teichmüller space for genus $g$, $\lambda$ is a partition of $k \le 2g-2$, $\rho$ is a non-elementary representation of $\pi_1(R)$ in $\PGL(2, \mathbb C)$ and $M_{\lambda, \rho}$ is the complex manifold of branched projective structures on a curve of genus $g$ with monodromy $\rho$ and whose branching points follow the partition $\lambda$.
Moreover, it is proved in \cite{francaviglia-ruffoni} that the condition is sufficient if the underlying complex curve is hyperelliptic and the divisor is reduced.
\end{remark}

\subsection{Branched Projective Structures Restricted to a Divisor}

To prove proposition \ref{proposition_universal_family_bps_over_bc} in its full generality, we need to introduce a technical notion of branched projective structure \emph{restricted} to a divisor.
It can be seen as a notion of jet of projective structure.
It can also be seen as a way to precise the information contained in a branching class in order to get an analog of proposition \ref{proposition_space_bps_given_bc} without the ambiguity in point (ii).

Let $X$ be a Riemann surface, along with an effective divisor $D$ and a branching class $a$ over $D$.
Let $D'$ be an effective divisor on $X$.

\begin{definition}
Let $U$ and $U'$ be two open neighborhoods of $D'$ in $X$. Write $V = U \cap U'$.
Let $p_1$, $p_2$ be branched projective structures, defined respectively on $U$ and $U'$, of divisor $D|_U$ (respectively $D|_{U'}$) and branching class $a|_U$ (respectively $a|_{U'}$). The projective structures $p_1$ and $p_2$ are said to \emph{coincide over the divisor $D'$} if the difference $p_2|_V - p_1|_V \in H^0\left(V, K_V^{\otimes 2}(-D|_V)\right)$ vanishes over the divisor $D'$, i.e. $p_2|_V-p_1|_V \in H^0\left(V, K_X^{\otimes 2}\left(-D|_V-D'|_V\right)\right)$.

A \emph{branched projective structure restricted to $D'$} of divisor $D$ and branching class $a$ is an equivalence class for the above equivalence relation.
We denote by $P_{X, D'}^a$ the set of branched projective structures of divisor $D$ and branching class $a$ restricted to $D'$.
\end{definition}

Any local holomorphic map defined around the points of $D'$ with branching class $a$ determines an element of $P_{X, D'}^a$.
Proposition \ref{proposition_space_bps_given_bc} implies the following~: 

\begin{proposition}\label{prop_struct_space_jets_bps}
Let $D$ be an effective divisor on a Riemann surface $X$, let $a \in A_X^D$ and $D'$ an effective divisor on $X$.
The space $P_{X, D'}^a$ is an affine space directed by $H^0\left(X, K_X^{\otimes 2}(-D)|_{D'}\right)$.
\end{proposition}

Now let $r \in P_{X, D'}^a$ be a branched projective structure on $X$ with branching class $a$ restricted to $D'$.
We denote by $\mathcal P_X^{a, r}$ the space of branched projective structures of branching class $a$ whose restriction to $D'$ is $r$.
The following proposition is a clear consequence of proposition \ref{proposition_space_bps_given_bc}.

\begin{proposition}\label{proposition_space_bps_given_rbps}
Let $D$ be an effective divisor on a Riemann surface $X$, $a \in A_X^D$ and $D'$ an effective divisor on $X$.
Let $r \in \mathcal P_{X, D'}^a$ be a restricted branched projective structure.
\begin{itemize}
\item[(i)] There exists an obstruction $\overline c(a, r) \in H^1\left(X, K_X^{\otimes 2}(-D-D')\right)$ such that $\overline c(a, r)=0$ if and only if $P_X^{a, r} \neq \emptyset$.
\item[(ii)] The set $P_X^{a, r}$, if nonempty, is an affine space directed by the vector space $H^0\left(X, K_X^{\otimes 2}(-D-D')\right)$.
\end{itemize}
\end{proposition}

In particular, if $D'$ is chosen so that $\deg(D) + \deg(D') > 4g-4$, $h^0\left(X, K_X^{\otimes 2}(-D-D')\right) = 0$ thus $P_X^{a, r}$ contains at most one point.

\subsection{Families of Restricted Branched Projective Structures}\label{subsection_families_rbps}

In this subsection we complete the proof of proposition \ref{proposition_universal_family_bps_over_bc}.
By a \emph{family of branching classes with auxiliary divisors of degree $n'$} we mean a family $(T, \mathcal X, \pi, \mathcal D, \mathfrak a, \mathcal D')$ where $(T, \mathcal X, \pi, \mathcal D, \mathfrak a)$ is a family of branching classes and $\mathcal D'$ is an effective divisor on $\mathcal X$ that contains no fiber of $\pi$, such that the restriction of $\mathcal D'$ to any fiber of $\pi$ has degree $n'$.

\begin{definition}
Let $\mathcal G = (T, \mathcal X, \pi, \mathcal D, \mathfrak a, \mathcal D')$ be a family of branching classes with auxiliary divisors.
Let $\mathcal U$ and $\mathcal U'$ be two open neighborhoods of $\mathcal D'$ in $\mathcal X$.
Denote $\mathcal V = \mathcal U \cap \mathcal U'$.
Let $\mathfrak p_1$, $\mathfrak p_2$ be relative branched projective structures, defined respectively on $\mathcal U$ and $\mathcal U'$, of divisor $\mathcal D|_{\mathcal U}$ (respectively $\mathcal D|_{\mathcal U'}$), and of branching class $\mathfrak a|_{\mathcal U}$ (respectively $\mathfrak a|_{\mathcal U'}$).
The projective structures $\mathfrak p_1$ and $\mathfrak p_2$ are said to \emph{coincide over the divisor $\mathcal D'$} if the difference $\mathfrak p_2|_{\mathcal V} - \mathfrak p_1|_{\mathcal V} \in H^0(\mathcal V, K_{\mathcal V/T}^{\otimes 2}(-\mathcal D|_{\mathcal V}))$ vanishes over the divisor $\mathcal D'$, i.e. $\mathfrak p_2|_{\mathcal V}-\mathfrak p_1|_{\mathcal V} \in H^0(\mathcal V, K_{\mathcal V/T}^{\otimes 2}(-\mathcal D|_{\mathcal V}- \mathcal D'))$.

A \emph{branched projective structure on $\mathcal X$, relative to $T$, restricted to $\mathcal D'$} of divisor $\mathcal D$ and branching class $\mathfrak a$ is an equivalence class for the above equivalence relation.
\end{definition}

\begin{definition}\label{definition_families_rbps}
\begin{itemize}
\item[(i)]
Let $n' \in \N$.
A \emph{family of restricted branched projective structures of auxiliary degree $n'$} is given by a family $\mathcal F = (S, \mathcal X, \pi, \mathcal D, \mathfrak a, \mathcal D', \mathfrak r)$ where $(S, \mathcal X, \pi, \mathcal D, \mathfrak a, \mathcal D')$ is a family of branching classes with auxiliary divisors of degree $n'$ and $\mathfrak r$ is a branched projective structure on $\mathcal X$ relative to $S$, restricted to $\mathcal D'$, with branching divisor $\mathcal D$ and branching class $\mathfrak a$.
\item[(ii)]
Let $\mathcal G$ be a family of branching classes with auxiliary divisors of basis $T$.
A \emph{family of restricted branched projective structures over $\mathcal G$} is a family of restricted branched projective structures $\mathcal F$ of basis $S$, along with a morphism $f : S \to T$ such that the underlying family of branching classes with auxiliary divisors is $f^*\mathcal G$.
\item[(iii)]
Let $S$ be an analytic space, $\mathcal F$ a family of restricted branched projective structures over $\mathcal G$ with basis $S$.
The family $\mathcal F$ is said to be \emph{universal} if any family $\mathcal F'$ of restricted branched projective structures over $\mathcal G$ is uniquely obtained as the pullback of $\mathcal F$ by some morphism $g : S' \to S$, where $S'$ is the basis of $\mathcal F'$.
\end{itemize}
\end{definition}

Fix $\mathcal G = (T, \mathcal X, \pi, \mathfrak a)$ a family of branching classes of genus $g$ and degree $n$, as well as an effective divisor $\mathcal D'$ that contains no fiber of $\mathcal G$.
Write $\mathcal H$ the corresponding family of branching classes with auxiliary divisors and $n'$ the auxiliary degree.

For $t \in T$, write respectively $X_t$, $D_t$, $a_t$ and $D'_t$ the curve, branching divisor, branching class and auxiliary divisor in the fiber of $\mathcal H$.
Consider
\begin{equation}
\mathcal P_{\mathcal X, \mathcal D'}^{\mathfrak a} = \left\{(t, r)| t \in T, r \in P_{X_t, D'_t}^{a_t}\right\}
\end{equation}
the union of the spaces of restricted projective structures on the fibers of $\mathcal H$.
There is an obvious map
\begin{equation}
\overline \alpha : \mathcal P_{\mathcal X, \mathcal D'}^{\mathfrak a} \to T
\end{equation}

The following lemma is a direct consequence of lemma \ref{lemma_relative_bps_schwarzian_derivative} :

\begin{lemma}
Let $\mathfrak r_0$ be a branched projective structure on $\mathcal X$, relative to $T$, restricted to $\mathcal D'$, of divisor $\mathcal D$ and relative branching class $\mathfrak a$.
\begin{itemize}
\item[(i)] If $\mathfrak r$ is another such structure, then the difference $\mathfrak r - \mathfrak r_0$ is a holomorphic section of the vector bundle $K_{\mathcal X/T}^{\otimes 2}(-\mathcal D)|_{\mathcal D'}$.
\item[(ii)] If $q \in H^0\left(\mathcal D', K_{\mathcal X/T}^{\otimes 2}(-\mathcal D)|_{\mathcal D'}\right)$ is a section, then there exists $\mathfrak r$ a branched projective structure on $\mathcal X$, relative to $T$, restricted to $\mathcal D'$, of relative branching class $\mathfrak a$, such that $\mathfrak r-\mathfrak r_0 = q$.
\end{itemize}
\end{lemma}

The dimension of $H^0\left(D'_t, K_{X_t}^{\otimes 2}\left(-D_t\right)|_{D'_t}\right)$ does not depend on $t \in T$.
Thus $\pi_*\left(K_{\mathcal X/T}^{\otimes 2}(-\mathcal D)|_{\mathcal D'}\right)$ is a vector bundle on $T$, and the datum of a section in $H^0\left(\mathcal D', K_{\mathcal X/T}^{\otimes 2}(-\mathcal D)|_{\mathcal D'}\right)$ is the same as the datum of a section in $H^0\left(T, \pi_*\left(K_{\mathcal X/T}^{\otimes 2}(-\mathcal D)|_{\mathcal D'}\right) \right)$.
As a consequence, $\mathcal P_{\mathcal X, \mathcal D'}^{\mathfrak a}$ has the structure of an affine bundle over $T$, whose holomorphic (local) sections are the ones defined by (local over $T$) families of branched projective structures on $\mathcal G$ restricted to $\mathcal D'$.

The pullback family of branching classes $\overline \alpha^*\mathcal G = \left(\mathcal P_{\mathcal X, \mathcal D'}^{\mathfrak a}, \overline \alpha^*\mathcal X, \overline \alpha^* \pi, \overline \alpha^* \mathcal D, \overline \alpha^* \mathfrak a\right)$ comes with a tautological relative branched projective structure $\mathfrak r_0$ restricted to $\overline \alpha^* \mathcal D'$~: if $(t, r) \in \mathcal P_{\mathcal X, \mathcal D'}^{\mathfrak a}$, the branched projective structure defined by $\mathfrak r_0$ on $\left(\overline \alpha^*\pi\right)^{-1}(t, r) = X_t$ with branching class $a_t$ is $r$.

The discussion around lemma \ref{lemma_univ_family_bc} applies \emph{mutatis mutandis}, to show that the family\break $\left(\mathcal P_{\mathcal X, \mathcal D'}^{\mathfrak a}, \overline \alpha^*\mathcal X, \overline \alpha^* \pi, \overline \alpha^* \mathcal D, \overline \alpha^* \mathfrak a, \overline \alpha^* \mathcal D', \mathfrak r_0\right)$, along with the morphism $\overline \alpha : \mathcal P_{\mathcal X}^{\mathfrak a} \to T$ is a universal family of branched projective structures over $\mathcal H$.
We thus have the following lemma :

\begin{lemma}\label{lemma_universal_family_rbps_over_bc}
If $\mathcal G$ is a family of branching classes with auxiliary divisors, there exists a universal family of restricted branched projective structures over $\mathcal G$.
\end{lemma}

Now let $\mathcal I = (T, \mathcal X, \pi, \mathcal D, \mathfrak a, \mathcal D', \mathfrak r)$ be a family of restricted branched projective structures with genus $g$, branching degree $n$ and auxiliary degree $n'$.
Suppose $n+n' > 4g-4$.
Denote by $r_t$ the restricted branched projective structure in $\mathcal I$ over $t \in T$.
Consider
\begin{equation}
\mathcal P_{\mathcal X}^{\mathfrak a, \mathfrak r} = \left\{(t, p)|t \in T, p \in P_{X_t}^{a_t, r_t}\right\}
\end{equation}
the set of branched projective structures whose restrictions are given by the fibers of $\mathcal I$.
There is an obvious map
\begin{equation}
\rho : \mathcal P_{\mathcal X}^{\mathfrak a, \mathfrak r} \to T
\end{equation}

According to proposition \ref{proposition_space_bps_given_rbps}, the fiber of $\rho$ over $t \in T$ is either empty or an affine space directed by $H^0\left(X_t, K_{X_t}^{\otimes 2}\left(-D_t-D_t'\right)\right)$.
Since moreover $\deg(D_t)+\deg(D_t')=n+n'>4g-4$, the line bundle $K_{X_t}^{\otimes 2}\left(-D_t-D_t'\right)$ has negative degree, thus has no nonzero section.
As a consequence, the fibers of $\alpha$ contain at most one point, and $\mathcal P_{\mathcal X}^{\mathfrak a, \mathfrak r}$ can be seen as a subset of $T$.

Similarly to subsection \ref{subsection_schwarzian_parametrization_in_family}, consider the sheaf $\mathcal B \mathcal R_{\mathcal X}^{\mathfrak a, \mathfrak r}$ of branched projective structures on $\mathcal X$, relative to $T$, branched on $\mathcal D$, of relative branching class $\mathfrak a$, whose restriction to $\mathcal D'$ is $\mathfrak r$.
This is a locally nonempty sheaf of affine spaces directed by $K_{\mathcal X/T}^{\otimes 2}(-\mathcal D-\mathcal D')$ and thus is associated to a cohomology class $\overline{\mathcal C}(\mathfrak a, \mathfrak r) \in H^1\left(X, K_{\mathcal X/T}^{\otimes 2}(-\mathcal D-\mathcal D')\right)$ that vanishes if and only if $\mathcal B \mathcal R_{\mathcal X}^{\mathfrak a, \mathfrak r}$ has global sections.
One also has an exact sequence similar to \eqref{eq_leray_exact_sequence}~:
\begin{equation}\label{eq_leray_exact_sequence_rbps}
H^1\left(T, \pi_*K_{\mathcal X/T}^{\otimes 2}(-\mathcal D-\mathcal D')\right) \to H^1\left(\mathcal X, K_{\mathcal X/T}^{\otimes 2}(-\mathcal D - \mathcal D')\right) \xrightarrow{\overline \chi_{\mathcal H}} H^0\left(T, R^1\pi_*K_{\mathcal X/T}^{\otimes 2}(-\mathcal D - \mathcal D')\right)
\end{equation}

Write
\begin{equation}
\overline{\mathfrak c}(\mathfrak a, \mathfrak r) = \overline \chi_{\mathcal H}\left(\overline{\mathcal C}(\mathfrak a, \mathfrak r)\right)
\end{equation}
Since $h^0\left(K_{X_t}^{\otimes 2}\left(-D_t-D_t'\right)\right)$ vanishes for any $t \in T$, Riemann-Roch formula implies that\break $h^1\left(K_{X_t}^{\otimes 2}\left(-D_t-D_t'\right)\right)$ does not depend on $t \in T$, so that $R^1\pi_*K_{\mathcal X/T}^{\otimes 2}(-\mathcal D - \mathcal D')$ is a vector bundle on $T$ of which $\overline{\mathfrak c}(\mathfrak a, \mathfrak r)$ is a holomorphic section.
For any $t \in T$, one has $\overline{\mathfrak c}(\mathfrak a, \mathfrak r) = \overline c(a_t, r_t)$.
Thus $\mathcal P_{\mathcal X}^{\mathfrak a, \mathfrak r}$ is the zero locus of the section $\overline{\mathfrak c}(\mathfrak a, \mathfrak r)$.

Let $\mathcal I' = \left(\mathcal P_{\mathcal X}^{\mathfrak a, \mathfrak r}, \rho^* \mathcal X, \rho^* \pi, \rho^* \mathcal D, \rho^* \mathfrak a, \rho^* \mathcal D', \rho^* \mathfrak r\right)$.
The map $\chi_{\mathcal I'}$, defined as in \eqref{eq_leray_exact_sequence_rbps}, is the zero map.
Moreover, if $T$ is Stein, the exact sequence \eqref{eq_leray_exact_sequence_rbps} has vanishing first term, which implies that the second term also vanishes as soon as $\chi_{\mathcal I}$ vanishes.
As a consequence, up to restricting $\mathcal P_{\mathcal X}^{\mathfrak a, \mathfrak r}$ to an open subset, there is a branched projective structure on $\rho^* \mathcal X$ relative to $\mathcal P_{\mathcal X}^{\mathfrak a, \mathfrak r}$ of branching class $\rho^* \mathfrak a$ and whose restriction to $\rho^*\mathcal D'$ is $\mathfrak r$.
Such a relative branched projective structure is unique since the fibers of $\rho$ contain at most one element.
We denote it by $\mathfrak p_0$.

A \emph{family of branched projective structures over the family of restricted branched projective structures $\mathcal I = (T, \mathcal X, \pi, \mathcal D, \mathfrak a, \mathcal D', \mathfrak r)$} is a family of branched projective structures $\mathcal F_1 = (T_1, \mathcal X_1, \pi_1, \mathcal D_1, \mathfrak a_1, \mathfrak p_1)$ along with a morphism $f : T \to T'$ such that the family of branching classes $(T_1, \mathcal X_1, \pi_1, \mathcal D_1, \mathfrak a_1)$ is the pullback by $f$ of $(T, \mathcal X, \pi, \mathcal D, \mathfrak a)$ and the restriction of $\mathfrak p_1$ to $f^*\mathcal D'$ is $f^* \mathfrak r$.
Such a family $\mathcal F_1$ is \emph{universal} if any other such family is uniquely obtained as a pullback of $\mathcal F_1$.

Clearly, the family of branched projective structures $\left(\mathcal P_{\mathcal X}^{\mathfrak a, \mathfrak r}, \rho^* \mathcal X, \rho^* \pi, \rho^* \mathcal D, \rho^* \mathfrak a, \mathfrak p_0\right)$, along with the morphism $\rho : \mathcal P_{\mathcal X}^{\mathfrak a, \mathfrak r} \to T$ is a universal family of branched projective structures over $\mathcal I$.
We have shown:

\begin{lemma}
Let $\mathcal I$ be a family of restricted branched projective structures of genus $g$, branching degree $n$ and auxiliary degree $n'$. If $n+n' > 4g-4$, there exists a universal family of branched projective structures over $\mathcal I$.
\end{lemma}

Now we are able to complete the proof of proposition \ref{proposition_universal_family_bps_over_bc}.
Let $\mathcal G$ be a family of branching classes of genus $g$ and branching degree $n$, of basis $T_{\mathcal G}$.
Take $\mathcal D'$ a divisor on the total space of $\mathcal G$ that contains no fiber of $\mathcal G$ and let $\mathcal H$ be the corresponding family of branching classes with auxiliary divisors.
Write $n'$ the auxiliary degree of $\mathcal H$, and choose $\mathcal D'$ such that $n+n' > 4g-4$.
Let $\mathcal I$ be the universal family of restricted projective structures over $\mathcal H$, $T_{\mathcal I}$ its basis and $f_{\mathcal G \mathcal I} : T_{\mathcal I} \to T_{\mathcal G}$ the associated morphism.
Now let $\mathcal F$, of basis $T_{\mathcal F}$, be the universal family of branched projective structures over $\mathcal I$.
Write $f_{\mathcal I \mathcal F} : T_{\mathcal F} \to T_{\mathcal I}$ the associated morphism.
The family $\mathcal F$, along with the morphism $f_{\mathcal G \mathcal F} = f_{\mathcal G \mathcal I} \circ f_{\mathcal I \mathcal F} : T_{\mathcal F} \to T_{\mathcal G}$ is a family of branched projective structures over $\mathcal G$.

Let $\mathcal F_1$, of basis $T_{\mathcal F_1}$ with a morphism $f_{\mathcal G \mathcal F_1} : T_{\mathcal F_1} \to T_{\mathcal G}$, be another family of branched projective structures over $\mathcal G$.
By restricting the relative branched projective structure of $\mathcal F_1$ to the pullback auxiliary divisor $f_{\mathcal G \mathcal F_1}^* \mathcal D'$, one gets a family of restricted branched projective structures $\mathcal I_1$, of basis $T_{\mathcal F_1}$, over the family $\mathcal H$, the associated morphism being $f_{\mathcal G \mathcal F_1}$.
By universality of $\mathcal I$, the family $\mathcal I_1$ is the pullback of $\mathcal I$ by some morphism $f_{\mathcal I \mathcal F_1}$ verifying $f_{\mathcal G \mathcal F_1} = f_{\mathcal G \mathcal I} \circ f_{\mathcal I \mathcal F_1}$.
As a consequence, $\mathcal F_1$ is a family of branched projective structures over $\mathcal I$, the associated morphism being given by $f_{\mathcal I \mathcal F_1}$.
By universality of $\mathcal F$, there exists a morphism $f_{\mathcal F \mathcal F_1} : T_{\mathcal F_1} \to T_{\mathcal F}$ such that $\mathcal F_1 = f_{\mathcal F \mathcal F_1}^* \mathcal F$ and $f_{\mathcal I \mathcal F_1} = f_{\mathcal I \mathcal F} \circ f_{\mathcal F \mathcal F_1}$ thus $f_{\mathcal G \mathcal F_1} = f_{\mathcal G \mathcal F} \circ f_{\mathcal F \mathcal F_1}$.

It remains to show that $f_{\mathcal F \mathcal F_1}$ is the unique such morphism.
Suppose given $f : T_{\mathcal F_1} \to T_{\mathcal F}$ such that $\mathcal F_1 = f^* \mathcal F$ and $f_{\mathcal G \mathcal F_1} = f_{\mathcal G \mathcal F} \circ f$.
The family of restricted branched projective structures $\mathcal I_1$ is obtained by pulling back $\mathcal I$ by the morphism $f \circ f_{\mathcal I \mathcal F}$.
By universality of $\mathcal I$, one has $f_{\mathcal I \mathcal F} \circ f = f_{\mathcal I \mathcal F_1}$.
Thus $\mathcal F_1$ is the pullback of $\mathcal F$ by $f$ \emph{as a family of branched projective structures over $\mathcal I$}.
Since $\mathcal F$ is a universal such family, $f = f_{\mathcal F \mathcal F_1}$.
This concludes the proof of proposition \ref{proposition_universal_family_bps_over_bc}.

\section{Holonomy and Singularities of the Moduli Space of Marked Branched Projective Structures in Genus at Least Two}\label{section_singular_points}

In this section, we prove the following theorem. Recall that if $n, g \in \N$, $\mathcal P_g(n)$ stands for the moduli space of isomorphism classes of branched projective of genus $g$ and branching degree $n$.

\begin{theorem}\label{theorem_singular_points}
Let $n, g \in \N$, with $g \ge 2$.
Let $p$ be a branched projective structure of genus $g$ and branching degree $n$ on a marked Riemann surface $X$.
Let $\rho \in \Hom\left(\pi_1(X), \PSL(2, \C)\right)$ be a monodromy representation for $g$.
The branched projective structure $p$ is a singular point of the moduli space $\mathcal P_g(n)$ if and only if the image $\Image \rho \subset \PSL(2, \C)$ of the monodromy representation $\rho$ has commutative image and the action of $\Image \rho$ on $\CP^1$ admits a fixed point.
\end{theorem}

\begin{remark}
\begin{itemize}
\item[(i)] It is equivalent to say that $p$ is a singular point of $\mathcal P_g(n)$ if and only if, up to conjugation, $\Image \rho$ is contained either in $\{z \mapsto \alpha z, \alpha \in \C^*\}$ or in $\{z \mapsto z + \alpha, \alpha \in \C\}$.
\item[(ii)] In particular, branched projective structures with nonelementary holonomy are smooth points of $\mathcal P_g(n)$. 
The holonomy of a branched projective structure whose degree is either odd or smaller than $2g-2$ is nonelementary (see \cite{gallo-kapovich-marden}).
Thus if $n$ is odd or $n < 2g-2$, then $\mathcal P_g(n)$ is smooth.
\item[(iii)] A branched projective structure in $\mathcal P_g(n)$ is a smooth point if and only if its holonomy is a smooth point of the analytic space $\Hom\left(\pi_1(X), \PGL(2, \C)\right)$ (see \cite{hubbard} for singularities of the latter), where $X$ is the underlying Riemann surface.
\item[(iv)] The holonomy $\rho \in \Hom\left(\pi_1(X), \PGL(2, \C)\right)$ of a branched projective structure $p$ is only defined up to conjugation with an element of $\PGL(2, \C)$.
However, the criterium on $\rho$ for $p$ to be singular does not depend on the choice of $\rho$.
\end{itemize}
\end{remark}

In order to study the singularities of $\mathcal P_g(n)$, we need to have an insight in its the tangent spaces.
We achieve this by studying the infinitesimal automorphisms of branched projective structures.

\subsection{Infinitesimal Automorphisms of a Branched projective Structure}\label{subsection_infinitesimal_automorphisms}

Let $X$ be a compact Riemann surface of genus $g \ge 2$ and $p$ a branched projective structure on $X$ of branching divisor $D$  and branching class $a \in A_X^D$.
Denote by $\Lambda_p$ the sheaf of \emph{infinitesimal automorphisms} of the branched projective structure $p$.
A local vector field $V$ defined on a open subset $U \subset X$ is a local section of $\Lambda_p$ if the flow of $V$ (defined at small times on relatively compact subsets of $U$) preserves the charts of $p$.
More precisely, if $\phi_V^t$ is the flow of $V$ at time $t \in \C$, then $V \in \Gamma\left(U, \Lambda_p\right)$ if and only if for any relatively compact open subset $U' \subset U$, any chart $f : U' \to \CP^1$ of $p$ and any small enough $t \in \C$, the map $f \circ \phi_V^t$ is again a chart of $p$.

The sheaf $\Lambda_p$ is a subsheaf of $T_X\left(-D^{\red}\right)$, where $D^{\red}$ is the reduced divisor associated to the branching divisor $D$.
A germ of holomorphic vector field $V_{(x)}$ at $x \in X$ is in the stalk of $\Lambda_p$ at $x$ if and only if $V_{(x)}$ is the pullback of a germ of projective vector field by some chart of $p$.

Let $\rho \in \Hom(\pi_1(X), \PSL(2, \C)$ be a monodromy representation of $p$.
Let $\Ad_p$ be the sheaf on $X$ whose stalk at $x \in X$ contains germs of meromorphic vector fields at $x$ that are obtained by pulling back a germ of projective vector field by a chart of the branched projective structure $p$.
In particular, $\Ad_p$ is a subsheaf of the sheaf of meromorphic functions on $X$ and $\Lambda_p$ is a subsheaf of $\Ad_p$.
The sheaf $\Ad_p$ is a local system whose stalks are $3$-dimensional Lie algebras modeled on $\mathfrak{sl}(2, \C)$, and whose holonomy is given by $\rho$, where $\PSL(2, \C)$ acts on its Lie algebra by the adjoint action.

Let $x \in X$ be a branched point of $p$ of degree $n$ and $z$ a local coordinate of $X$ centered at $x$ such that $f(z) = z^{n+1} \in \CP^1 = \C \cup \{\infty\}$ is a chart of $p$.
The pullback by $f$ of the vector field $W = \left(a_0 + a_1 \zeta + a_2 \zeta^2\right) \partial_{\zeta}$, where $\zeta$ is the classical coordinate on $\CP^1 = \C \cup \{\infty\}$, is given by
\begin{equation}
f^*W = \left(a_0\frac 1 {z^n} + a_1 z + a_2 z^{n+2}\right) \partial_z
\end{equation}
This implies that the quotient $\Ad_p / \Lambda_p$ is supported on the branched points of $p$ and the stalk of $\Ad_p / \Lambda_p$ at any branched point of $p$ is a complex line, whatever the branching degree.
The short exact sequence
\begin{equation}
0 \to \Lambda_p \to \Ad_p \to \Ad_p / \Lambda_p \to 0
\end{equation}
gives in cohomology the long exact sequence (recall that $X$ has no infinitesimal automorphism, so $H^0\left(X, \Lambda_p\right) = 0$)
\begin{equation}\label{eq_exact_autom_local_syst}
\begin{aligned}
&0 \to H^0\left(X, \Ad_p\right) \to H^0\left(X, \Ad_p / \Lambda_p\right) \to H^1\left(X, \Lambda_p\right) \to H^1\left(X, \Ad_p\right) \to 0 \to \\
&  \to  H^2\left(X, \Lambda_p\right) \to H^2\left(X, \Ad_p\right) \to 0
\end{aligned}
\end{equation}
In particular, the injection $\Lambda_p \to \Ad_p$ induces an isomorphism $H^2\left(X, \Lambda_p\right) \simeq H^2\left(X, \Ad_p\right)$.

The sheaf cohomology of the local system $\Ad_p$ is the same as its cohomology with local coefficients, see \cite{steenrod}.
Moreover, Poincaré duality provides an isomorphism $H^2\left(X,  \Ad_p\right) \simeq H_0\left(X, \Ad_p\right)$, where the latter group is the homology with local coefficients of $\Ad_p$, see \cite{sun} and references therein.
We have the following lemma, see \cite{whitehead}.
\begin{lemma}
Let $L_0 \subset \mathfrak g$ be the linear subspace generated by the family $\left(\rho(\gamma)\cdot W - W\right)_{\begin{subarray}{1}\gamma \in \pi_1(X) \\ W \in \mathfrak g \end{subarray}}$.
Let $L_1 \subset \mathfrak g$ be the linear subspace of fixed points of the action on $\pi_1(X)$ on $\mathfrak g$ : $L_1 = \{W \in \mathfrak g | \forall \gamma \in \pi_1(X), \rho(\gamma) \cdot W = W\}$.
One has
\begin{itemize}
\item[(i)] $H_0\left(X, \Ad_p\right) = \mathfrak g / L_0$
\item[(ii)] $H^0\left(X, \Ad_p\right) = L_1$
\end{itemize}
\end{lemma}
Take $W \in \mathfrak g$.
Recall that $\mathfrak g$ is the Lie algebra of projective vector fields, so $W$ can be written $W = \left(a_0 + a_1 \zeta + a_2 \zeta^2 \right) \partial_{\zeta}$.
A computation shows that if $g \in G$ is the Möbius transformation $g : \zeta \mapsto \frac{\alpha \zeta + \beta}{\gamma \zeta + \delta}$, then
\begin{itemize}
\item[(i)] If $\alpha = \delta = 1$, $\gamma = 0$, i.e. $g$ is parabolic fixing $\infty \in \CP^1$, then $g \cdot W = W + \bigl((a_2 \beta^2 - a_1 \beta) - 2a_2 \beta \zeta\bigr) \partial_{\zeta}$
\item[(ii)] If $\beta = \gamma = 0$, $\delta = 1$, i.e. $g$ is either elliptic or loxodromic fixing $0$ and $\infty$ in $\CP^1$, then $g \cdot W = W + \bigl(a_0(\alpha-1) + a_2(\frac{1}{\alpha} - 1)\zeta^2\bigr)\partial_{\zeta}$
\end{itemize}
As a consequence, there are three cases for the dimensions of the spaces in equation \eqref{eq_exact_autom_local_syst}.
Here we denote by $r$ the number of branched points of $p$, without taking the multiplicity into account, so that $h^0\left(X, \Ad_p / \Lambda_p\right) = r$.
Recall that the Euler characteristic of $X$ is given by the alternate sum of the betti numbers of the cohomology of any local system on $X$, so in particular $3(2-2g) = h^0\left(X, \Ad_p\right) - h^1\left(X, \Ad_p\right) + h^2\left(X, \Ad_p\right)$ (here the factor $3$ stands for the dimension of the model space for coefficients $\mathfrak g$). 
\begin{itemize}
\item[(i)] If the image of $\rho$ is the trivial subgroup of $G$, $L_0 = \{0\}$, $L_1 = \mathfrak g$, so $h^0\left(X, \Ad_p\right) = 3$, and $h^2\left(X, \Lambda_p\right) = h^2\left(X, \Ad_p\right) = h_0\left(X, \Ad_p\right) = 3$.
As a consequence $h^1\left(X, \Ad_p\right) = 6g$, and $h^1\left(X, \Lambda_p\right) = 6g-3+r$.
\item[(ii)] Suppose the image $\Image(\rho) \subset G$ is not the trivial subgroup but all the nontrivial elements of $\Image(\rho)$ have the same set of fixed points.
More precisely, the nontrivial elements of $\Image(\rho)$ are either all parabolic transformations with the same fixed point (conjugate to $z \mapsto z + \beta$), or all loxodromic or elliptic tranformations with the same two fixed points (conjugate to $z \mapsto \alpha z$).
Then $h^0\left(X, \Ad_p\right) = 1$, and $h^2\left(X, \Lambda_p\right) = h^2\left(X, \Ad_p\right) = h_0\left(X, \Ad_p\right) = 1$.
As a consequence $h^1\left(X, \Ad_p\right) = 6g-4$, and $h^1\left(X, \Lambda_p\right) = 6g-5+r$.
\item[(iii)] In the other cases, i.e. if the image $\Image(\rho) \subset G$ contains at least two nontrivial transformations $g_1$ and $g_2$ with $g_1$ not fixing a fixed point of $g_2$, then $h^0\left(X, \Ad_p\right) = 0$, and $h^2\left(X, \Lambda_p\right) = h^2\left(X, \Ad_p\right) = h_0\left(X, \Ad_p\right) = 0$.
As a consequence $h^1\left(X, \Ad_p\right) = 6g-6$, and $h^1\left(X, \Lambda_p\right) = 6g-6+r$.
\end{itemize}

\begin{remark}
The group $h^2\left(X, \Lambda_p\right)$ is nonzero if and only if the monodromy representation $\rho$ is a singular point of $\Hom\left(\pi_1(X), G\right)$ (see \cite{hubbard}).
\end{remark}

\subsection{Tangent Spaces of the Universal Spaces of Divisors}\label{subsection_tangent_spaces_univ_spaces_div}

Let $n, g \in \N$ with $g \ge 2$.
Following subsection \ref{subsection_families_bps}, let $\mathcal U_g = \left(\mathcal T_g, \mathcal X_g, \pi_g, m_g\right)$ be the universal family of marked Riemann surfaces of genus $g$ ($m_g$ is the marking).
In particular $\mathcal T_g$ is the Teichmüller space.
Let $\mathcal U_g^{\Div}(n) = \left(\mathcal T_g(n), \mathcal X_g^{\Div}(n), \pi_g^{\Div}(n), m_g^{\Div}(n), \mathcal D_g^{\Div}(n)\right)$ be the universal family of marked curves of genus $g$ with divisors of degree $n$, and let $\gamma_g(n) : \mathcal T_g(n) \to \mathcal T_g$ be the forgetful morphism.
In particular, $\mathcal T_g(n)$ is the moduli space of marked Riemann surfaces with divisors.

A \emph{partition} of $n$ is an ordered finite sequence of stictly positive integers $(k_1 \le k_2 \le \cdots \le k_r)$ with $k_1 + \cdots + k_r = n$.
Denote by $\Part(n)$ the set of partitions of $n$.
Let $(X, D, m) \in \mathcal T_g(n)$~: $X$ is a Riemann surface, $m$ is a marking on $X$, and $D$ is a divisor of degree $n$ on $X$.
The divisor $D$ is given by $D = k_1 \cdot x_1 + \cdots + k_r \cdot x_r$ with $k_i \in \N^*$, $x_i \in X$ pairwise distinct and $k_1 + \cdots + k_r = n$.
The $k_i$'s can be chosen to be ordered, so that the sequence $(k_1, \dots, k_r)$ is a partition of $n$, uniquely determined by $D$.

In this way we define a map $\varpi : \mathcal T_g(n) \to \Part(n)$.
For any partition $\kappa = (k_1 \le \dots \le k_r)$ of $n$, $\varpi^{-1}(\kappa)$ is a smooth $3g-3+r$-dimensional analytic subspace of $\mathcal T_g(n)$ (see \cite{barlet-magnusson} for details on the structure of the symmetric powers of a smooth curve).
Let us denote this space by $\mathcal T_g(\kappa) = \varpi^{-1}(\kappa)$, and by $\mathcal U_g(\kappa) = \left(\mathcal T_g(\kappa), \mathcal X_g^{\Div}(\kappa),  \mathcal D_g^{\Div}(\kappa), m_g^{\Div}(\kappa), \pi_g^{\Div}(\kappa)\right)$ the restriction of the family $\mathcal U_g(n)$ to the basis $\mathcal T_g(\kappa)$.
Let us also denote by $\gamma_g(\kappa) : \mathcal T_g(\kappa) \to \mathcal T_g$ the forgetful submersion, restriction of $\gamma_g(n)$.
Finally, denote by $\Sym_{\kappa}(X)$ the submanifold of the symmetric power $\Sym^n(X)$ containing the divisors whose associated partition is $\kappa$.

Let us now focus on the space $\mathcal T_g(\kappa)$ from the point of view of infinitesimal deformations of marked curves with divisor.
Let $(X, D, m) \in \mathcal T_g(\kappa)$ be a smooth marked complex curve with divisor ($m$ is the marking on $X$).
One has a first exact sequence:
\begin{equation}\label{eq_exact_tan_teich}
0 \to \left(T_{\Sym_{\kappa}(X)}\right)_D \to \left(T_{\mathcal T_g(\kappa)}\right)_{(X, D, m)} \to \left(T_{\mathcal T_g}\right)_{(X, m)} \to 0
\end{equation}
and another one, denoting by $D^{\red}$ the reduced divisor on $X$ with the same support as $D$ :
\begin{equation}
0 \to T_X\left(-D^{\red}\right) \to T_X \to T_X|_{D^{\red}} \to 0
\end{equation}
Since $g \ge 2$, one has in cohomology :
\begin{equation}
0 \to H^0\left(X, T_X|_{D^{\red}}\right) \to H^1\left(X, T_X\left(-D^{\red}\right)\right) \to H^1\left(X, T_X\right) \to 0
\end{equation}

There is an obvious isomorphism $\iota : \left(T_{\Sym_{\kappa}(X)}\right)_D \xrightarrow{\sim} H^0\left(X, T_X|_{D^{\red}}\right)$.
One also has an isomorphism $KS : \left(T_{\mathcal T_g}\right)_{(X, m)} \xrightarrow{\sim} H^1\left(X, T_X\right)$, called the Kodaira-Spencer isomorphism, of which we recall the definition.

Consider the exact sequence
\begin{equation}\label{eq_kodaira_spencer_exact_sequence}
0 \to T_{\mathcal X_g/\mathcal T_g} \to T_{\mathcal X_g} \xrightarrow{d\pi_g} \pi_g^*T_{\mathcal T_g} \to 0
\end{equation}
By applying the left exact functor $\pi_{g*}$, it gives a morphism of vector bundles $\pi_{g*}\pi_g^*T_{\mathcal T_g} = T_{\mathcal T_g} \to R^1\pi_{g*}T_{\mathcal X_g/\mathcal T_g}$. 
The map $KS$ is the former isomorphism restricted to the fiber above $(X, m)$ (the equality $\pi_{g*}\pi_g^* T_{\mathcal T_g} = T_{\mathcal T_g}$ stands because $\pi_g$ has compact fibers).
It is a property of the Teichmüller space that the map $KS$ is an isomorphism (see for instance \cite{earle-eells}).

Let us denote by $\mathcal D_g^{\Div}(\kappa)$ the restriction of the divisor $\mathcal D_g^{\Div}(n)$ to the manifold $\mathcal X_g^{\Div}(\kappa)$.
The reduced divisor associated to  $\mathcal D_g^{\Div}(\kappa)$, denoted by $\mathcal D_g^{\Div}(\kappa)^{\red}$, is a complex manifold (a smooth analytic subspace), that is transverse to the fibers of $\pi_g^{\Div}(\kappa)$.
As a consequence, when writing the Kodaira-Spencer exact sequence \eqref{eq_kodaira_spencer_exact_sequence} for the family $\mathcal U_g^{\Div}(\kappa)$, we can restrict to the vector fields on $\mathcal X_g^{\Div}(\kappa)$ that are tangent to $\mathcal D_g^{\Div}(\kappa)^{\red}$ :

\begin{equation}
0 \to T_{\mathcal X_g^{\Div}(\kappa)/\mathcal T_g(\kappa)}\left(-\mathcal D_g^{\Div}(\kappa)^{\red}\right) \to T_{\mathcal X_g^{\Div}(\kappa)}\left(-\log\mathcal D_g^{\Div}(\kappa)^{\red}\right) \xrightarrow{d\pi_g^{\Div}(\kappa)} \pi_g^{\Div}(\kappa)^*T_{\mathcal T_g(\kappa)} \to 0
\end{equation}

Recall that $T_{\mathcal X_g^{\Div}(\kappa)}\left(-\log\mathcal D_g^{\Div}(\kappa)^{\red}\right)$ is the sheaf of vector fields on $\mathcal X_g^{\Div}(\kappa)$ that are tangent to $\mathcal D_g^{\Div}(\kappa)^{\red}$.
Note that the reason why, in the above diagram, the arrow $d\pi_g^{\Div}(\kappa)$ is surjective is that $\mathcal D_g^{\Div}(\kappa)^{\red}$ is a smooth analytic subspace transverse to the fibers of $d\pi_g^{\Div}(\kappa)$.
By applying the left exact functor $\pi_g^{\Div}(\kappa)_*$, one obtains in particular a morphism of vector bundles $T_{\mathcal T_g(\kappa)} \to R^1\pi_g^{\Div}(\kappa)_*T_{\mathcal X_g^{\Div}(\kappa)/\mathcal T_g(\kappa)}\left(-\mathcal D_g^{\Div}(\kappa)^{\red}\right)$.
By looking at that morphism at the point $(X, D, m) \in \mathcal T_g(\kappa)$, one obtains a linear map

\begin{equation}\label{eq_kodaira-spencer_2}
KS^{\Div} : \left(T_{\mathcal T_g(\kappa)}\right)_{(X, D, m)} \to H^1\left(X, T_X\left(-D^{\red}\right)\right)
\end{equation}

It is easy to see that the following diagram commutes :

\begin{equation}
\begin{tikzcd} 0 \arrow[r] &\left(T_{\Sym_{\kappa}(X)}\right)_D \arrow[d, "\iota"] \arrow[r] & \left(T_{\mathcal T_g(\kappa)}\right)_{(X, D, m)} \arrow[d, "KS^{\Div}"] \arrow[r] & \left(T_{\mathcal T_g}\right)_{(X, m)} \arrow[d, "KS"] \arrow[r] & 0 \\ 0 \arrow[r] & H^0\left(X, T_X|_{D^{\red}}\right) \arrow [r] & H^1\left(X, T_X\left(-D^{\red}\right)\right) \arrow[r] & H^1\left(X, T_X\right) \arrow[r]& 0 \end{tikzcd}
\end{equation}

By the five lemma, we have :

\begin{lemma}\label{lemma_KS'_iso}
The map $KS^{\Div}$ in \eqref{eq_kodaira-spencer_2} is an isomorphism.
\end{lemma}

\subsection{Tangent Spaces of the Universal Spaces of Branching Classes}\label{subsection_tangent_spaces_univ_spaces_bc}

Let $X$ be a Riemann surface of genus $g \ge 2$, $D$ a divisor on $X$ and $a \in A_X^D$ a branching class.
Let $U$ be an open set in $X$ and let $V \in\Gamma(U, T_X(-D^{\red}))$.
For any $t \in \mathbb R$ near $0$, denote by $\phi^t_V$ the flow of $V$ at time $t$.
Write
$$\chi_{a}(V) = \lim_{t \to 0} \frac{(\phi^t_V)^*a - a}t \in \Gamma(U, K_X|_D)$$
where the pullback of branching classes is the pullback of jets of holomorphic maps (thus the pullback of holomorphic maps).

Let $\zeta$ be a nonconstant holomorphic map on $U$, whose jets at the points of $D|_U$ are representatives of the branching class $a|_U$. In particular, $\zeta$ has branching divisor $D|_U$.
Write $V = f(\zeta)\partial_{\zeta}$, where $f$ is a multivalued holomorphic function.
A straightforward computation shows that $\chi_{a}(V) = f''(\zeta)d\zeta$.
As a consequence, if $D|_U = nx$ with $x \in U$ and if $z$ is a local coordinate centered at $x$ and such that $\zeta = z^{n+1}$, then if $V = (a_1z + a_2z^2 + \cdots)\partial_z$, the section $\chi_{a}(V) \in \Gamma(U, K_X|_D)$ is given in the coordinate $z$ by :
\begin{equation}\label{eq_formule_derivee_classe}
\chi_{a}(V) = \left((n+2)a_2 + 2(n+3)a_3z + 3(n+4)a_4z^2 + \cdots + n(2n+1)a_{n+1}z^{n-1}\right)dz
\end{equation}
In particular, the morphism of sheaves $\chi_{a} : T_X(-D^{\red}) \to K_X|_D$ is surjective.
Denoting by $\Gamma_{a}$ its kernel, one gets an exact sequence:

\begin{equation}\label{definition_gamma_alpha}
0 \to \Gamma_{a} \to T_X\left(-D^{\red}\right) \xrightarrow{\chi_a} K_X|_D \to 0
\end{equation}
and thus in cohomology
\begin{equation}\label{eq_exact_cohom_classes}
0 \to H^0(X, K_X|_D) \to H^1(X, \Gamma_{a}) \to H^1\left(X, T_X\left(-D^{\red}\right)\right) \to 0
\end{equation}

Let $n \in \N$.
Following subsection \ref{subsection_families_bps}, let $\mathcal U_g^{\BC} = \left(\mathcal A_g(n), \mathcal X_g^{\BC}(n), \pi_g^{\BC}(n), m_g^{\BC}(n), \mathcal D_g^{\BC}(n), \mathfrak a_g^{\BC}(n)\right)$ be the universal family of marked branching classes of genus $g$ and degree $n$.
Write $\delta_g(n) : \mathcal A_g(n) \to \mathcal T_g(n)$ the forgetful morphism.
According to subsection \ref{subsection_families_branching_classes}, $\mathcal A_g(n)$ is an affine bundle over $\mathcal T_g(n)$.

For any partition $\kappa$ of the integer $n$, denote by $\mathcal A_g(\kappa)$ the bundle $\mathcal A_g(n) \twoheadrightarrow \mathcal T_g(n)$ restricted to the basis $\mathcal T_g(\kappa) \subset \mathcal T_g(n)$.
Let $\mathcal U_g^{\BC}(\kappa) = \left(\mathcal A_g(\kappa), \mathcal X^{\BC}_g(\kappa), \mathcal D_g^{\BC}(\kappa), \pi_g^{\BC}(\kappa), \mathfrak a_g^{\BC}(\kappa)\right)$ be the pullback by the inclusion $\mathcal A_g(\kappa) \to \mathcal A_g(n)$ of the family $\mathcal U_g^{\BC}(n)$.
Denote by $\delta_g(\kappa)$ the restriction of $\delta_g(n)$ to $\mathcal A_g(\kappa)$.

Let $(X, D, a, m) \in \mathcal A_g(n)$ be a marked curve with divisor of degree $n$ and branching class.
Denote by $\kappa$ the partition of $n$ associated to the divisor $D$.
It is a consequence of proposition \ref{proposition_affine_structure_branching_classes} that the tangent space of the fiber $A_X^D = \delta_g(\kappa)^{-1}(X, D, m)$ at $a$ is $\left(T_{A_X^D}\right)_{a} = H^0(X, K_X|_D)$. 
Moreover we saw in section \ref{subsection_tangent_spaces_univ_spaces_div} that there is a Kodaira-Spencer isomorphism $KS^{\Div} : \left(T_{\mathcal T_g(\kappa)}\right)_{(X, D, m)} \xrightarrow{\sim} H^1\left(X, T_X\left(-D^{\red}\right)\right)$.
As a consequence, similarly to section \ref{subsection_tangent_spaces_univ_spaces_div}, one has identifications between the kernels and cokernels of the exact sequences \eqref{eq_exact_cohom_classes} and
\begin{equation}
0 \to \left(T_{A_X^D}\right)_a \to \left(T_{\mathcal A_g(\kappa)}\right)_{(X, D, a, m)} \to \left(T_{\mathcal T_g(\kappa)}\right)_{(X, D, m)} \to 0
\end{equation}

Let $\mathcal K$ (respectively $\mathcal K^{\rel}$) be the subsheaf of $T_{\mathcal X^{\BC}_g(\kappa)}$ (respectively $T_{\mathcal X^{\BC}_g(\kappa)/\mathcal A_g(\kappa)}$) that contains the vector fields preserving the relative branched atlas $\mathfrak a_g^{\BC}(\kappa)$.
The Kodaira-Spencer exact sequence for the family $\mathcal X^{\BC}_g(\kappa)$, when restricted to the vector fields preserving $\mathfrak a_g^{\BC}(\kappa)$, is given by :

\begin{equation}
0 \to \mathcal K^{\rel} \to \mathcal K \to \pi_g^{\BC}(\kappa)^* T_{\mathcal A_g(\kappa)} \to 0
\end{equation}

Applying the left exact functor $\pi_g^{\BC}(\kappa)_*$, one gets a morphism of sheaves $T_{\mathcal A_g(\kappa)} \to R^1 \pi_g^{\BC}(\kappa)_* \mathcal K^{\rel}$. In particular, at the point $(X, D, a, m)$, this morphism gives a linear map :

\begin{equation}\label{eq_kodaira_spencer_3}
KS^{\BC} : \left(T_{\mathcal A_g(\kappa)}\right)_{(X, D, a, m)} \to H^1\left(X, \Gamma_a\right)
\end{equation}

It is easy to see that the following diagram is commutative :

\begin{equation}
\begin{tikzcd} 0 \arrow[r] &\left(T_{A_X^D}\right)_a \arrow[d, equal] \arrow[r] & \left(T_{\mathcal A_g(\kappa)}\right)_{(X, D, a, m)} \arrow[d, "KS^{\BC}"] \arrow[r] & \left(T_{\mathcal T_g(\kappa)}\right)_{(X, D, m)} \arrow[d, "KS^{\Div}"] \arrow[r] & 0 \\ 0 \arrow[r] & H^0\left(X, K_X|_D\right) \arrow [r] & H^1\left(X, \Gamma_a\right) \arrow[r] & H^1\left(X, T_X\left(-D^{\red}\right)\right) \arrow[r]& 0 \end{tikzcd}
\end{equation}

Since $KS^{\Div}$ is an isomorphism by lemma \ref{lemma_KS'_iso}, we have the following lemma :

\begin{lemma}
The map $KS^{\BC}$ in \eqref{eq_kodaira_spencer_3} is an isomorphism.
\end{lemma}

\subsection{Tangent Spaces of the Universal Spaces of Branched Projective Structures}\label{subsection_tangent_spaces_univ_spaces_bps}

Let $X$ be a Riemann surface of genus $g \ge 2$, $D$ a divisor on $X$, $a \in A_X^D$ a branching class and $p \in P_X^a$ a branched projective structure.
Let $U$ be an open subset of $X$ and let $V \in\Gamma(U, \Gamma_a)$, where $\Gamma_a$ is defined in \eqref{definition_gamma_alpha}.
For any $t \in \R$ that is near $0$, the flow of $V$ at time $t$ is denoted by $\phi^t_V$.
Note that, by definition of $\Gamma_a$, $V \in \Gamma_a$ means that for all $t$ near $0$ one has $\left(\phi_V^t\right)^* a - a = 0$. Write
\begin{equation}
\xi_{p}(V) = \lim_{t \to 0} \frac{(\phi^t_V)^*p - p}t \in \Gamma\left(U, K_X^{\otimes 2}(-D)\right)
\end{equation}

Let $\zeta$ be a chart on $U$ of the branched projective structure $p$.
Write $V = f(\zeta)\partial_{\zeta}$, where $f$ is a multivalued holomorphic function.
A straightforward computation shows that $\xi_p(V) = f'''(\zeta)d\zeta^{\otimes 2}$.
Suppose $D|_U = nx$ with $x \in U$ and let $z$ be a coordinate centered at $x$ such that $\zeta = z^{n+1}$.
Write $V = (a_1z + a_2z^2 + \cdots)\partial_z$.
Then $$\xi_p(V) = \left(\sum_{k \ge 1} k(k + n+1)(k - (n+1))a_{k+1}z^{k-2}\right)dz^{\otimes 2}$$
According to equation \eqref{eq_formule_derivee_classe}, $V$ is a section of $\Gamma_a$ if and only if $a_2 = a_3 = \dots = a_{n+1} = 0$. As a consequence :
\begin{equation}\label{eq_xi_p}
\xi_p(V) = \left(\sum_{l \ge n}(l+2)(l+n+3)(l-n+1)a_{l+3}z^l\right)dz^{\otimes 2}
\end{equation}

Since for any $l \ge n$, $(l+2)(l+n+3)(l-n+1) \neq 0$, the morphism of sheaves $\xi_p : \Gamma_a \to K_X^{\otimes 2}(-D)$ is surjective.
Denote by $\Lambda_p$ its kernel.
It is the sheaf of infinitesimal automorphisms of $p$.
There is an exact sequence :

\begin{equation}
0 \to \Lambda_p \to \Gamma_a \xrightarrow{\xi_p} K_X^{\otimes 2}(-D) \to 0
\end{equation}
which gives, by taking cohomologies, an exact sequence :
\begin{equation}\label{eq_exact_cohom_struct}
0 \to H^0\left(X, K_X^{\otimes 2}(-D)\right) \to H^1(X, \Lambda_p) \to H^1\left(X, \Gamma_a\right) \to H^1\left(X, K_X^{\otimes 2}(-D)\right) \to H^2(X, \Lambda_p) \to 0
\end{equation}

For any $n \in \N$, according to subsection \ref{subsection_families_bps}, write $\mathcal U_g^{\BPS}(n) = \bigl(\mathcal P_g(n), \mathcal X_g^{\BPS}(n), \pi_g^{\BPS}(n), m_g^{\BPS}(n),\break \mathcal D_g^{\BPS}(n), \mathfrak a_g^{\BPS}(n), \mathfrak p_g^{\BPS}(n)\bigr)$ the universal family of marked branched projective structures of genus $g$ and degree $n$.
In particular $\mathcal P_g(n)$ is the moduli space of marked branched projective structures of degree $n$ and genus $g$.
Denote by $\alpha_g(n) : \mathcal P_g(n) \to \mathcal A_g(n)$ the forgetful map.

\subsubsection{Case $n < 2g-2$ or $n > 4g-4$}\label{subsubsection_singular_points_small_high_degree}

Let us first handle the case, corresponding to the framework of proposition \ref{proposition_univ_family_bps_low_high_deg}, where the branching degree $n$ and the genus $g$ verify $n < 2g-2$ or $n > 4g-4$.

In this subsubsection we prove the following proposition, which, combined with the study of subsection \ref{subsection_infinitesimal_automorphisms}, shows theorem \ref{theorem_singular_points}.

\begin{proposition}\label{proposition_singular_points_small_high_degree}
Take $g \ge 2$ and $n \in \N$.
\begin{itemize}
\item[(i)] If $n < 2g-2$, the analytic space $\mathcal P_g(n)$ is smooth.
\item[(ii)] If $n > 4g-4$ and $p \in \mathcal P_g(n)$, $p$ is a smooth point of $\mathcal P_g(n)$ if and only if its sheaf $\Lambda_p$ of infinitesimal automorphisms satisfies $H^2\left(X, \Lambda_p\right) = 0$.
\end{itemize}
\end{proposition}

\begin{proof}
In the case $n<2g-2$, according to proposition \ref{proposition_univ_family_bps_low_high_deg}, $\mathcal P_g(n)$ is an affine bundle on the smooth variety $\mathcal A_g(n)$, thus it is smooth.
The case $n > 4g-4$ is more subtle, it is the reason why we introduced the Kodaira-Spencer machinery.

From now on, suppose $n > 4g-4$.
Consider the vector bundle on $\mathcal A_g(n)$ whose fiber over a point $(X, D, a, m)$ is $H^1\left(X, K_X^{\otimes 2}(-D)\right)$, and denote by $\mathfrak c_g(n) = \mathfrak c\left(\mathfrak a_g^{\BC}(n)\right)$ its global section defined in \eqref{eq_section_obstruction}.
By proposition \ref{proposition_univ_family_bps_low_high_deg}, the map $\alpha_g(n)$ is an embedding and thus identifies $\mathcal P_g(n)$ with an analytic subspace of $\mathcal A_g(n)$.
More precisely, we saw in the proof of proposition \ref{proposition_univ_family_bps_low_high_deg} that the subspace $\mathcal P_g(n) \subset \mathcal A_g(n)$ is given by the zeroes of the section $\mathfrak c_g(n)$.
In particular, if $(X, D, a, m) \in \mathcal P_g(n) \subset \mathcal A_g(n)$, one has $\mathfrak c_g(n)(X, D, a, m) = 0$.
In that case, the differential $\left(d \mathfrak c_g(n) \right)_{(X, D, a, m)}$ is a linear map from $\left(T_{\mathcal A_g(n)}\right)_{(X, D, a, m)}$ to $H^1\left(X, K_X^{\otimes 2}(-D)\right)$.
We will show that this differential is surjective for any $p \in \mathcal A_g(n)$ such that $H^2\left(X, \Lambda_p\right)=0$, which implies smoothness of $\mathcal P_g(n)$ at $p$.

For any partition $\kappa$ of the integer $n$, denote by $\mathcal P_g(\kappa)$ the restriction of the bundle $\mathcal P_g(n) \twoheadrightarrow \mathcal A_g(n)$ to the basis $\mathcal A_g(\kappa) \subset \mathcal A_g(n)$.
Let $\mathcal U_g^{\BPS}(\kappa) = \left(\mathcal P_g(\kappa), \mathcal X^{\BPS}_g(\kappa), \pi_g^{\BPS}(\kappa), m_g^{\BPS}(\kappa), \mathcal D_g^{\BPS}(\kappa), \mathfrak a_g^{\BPS}(\kappa)\right)$ be the pullback of the family $\mathcal U_g^{\BPS}(n)$  by the inclusion $\mathcal P_g(\kappa) \to \mathcal P_g(n)$.
Denote by $\alpha_g(\kappa)$ the restriction of $\alpha_g(n)$ to $\mathcal P_g(\kappa)$.

Let $(X, D, a, p, m) \in \mathcal A_g(n)$ be a marked curve with divisor of degree $n$, branching class and branched projective structure.
Let $\kappa$ be the partition of $n$ associated to the divisor $D$.
Proposition \ref{proposition_space_bps_given_bc} implies that the tangent space of the fiber $P_X^{a} = \alpha_g(\kappa)^{-1}(X, D, a, m)$ at $p$ is $\left(T_{P_X^a}\right)_p = H^0\left(X, K_X^{\otimes 2}(-D)\right)$.
Moreover, as we saw in section \ref{subsection_tangent_spaces_univ_spaces_bc}, there is a Kodaira-Spencer isomorphism $KS^{\BC} : \left(T_{\mathcal A_g(\kappa)}\right)_{(X, D, a, m)} \xrightarrow{\sim} H^1\left(X, \Gamma_a\right)$.
Thus the first (respectively the third) terms of the two exact sequences \eqref{eq_exact_cohom_struct} and
\begin{equation}
0 \to \left(T_{P_X^a}\right)_p \to \left(T_{\mathcal P_g(\kappa)}\right)_{(X, D, a, p, m)} \to \left(T_{\mathcal A_g(\kappa)}\right)_{(X, D, a, m)} \xrightarrow{d\mathfrak c_g(\kappa)_{(X, D, a, m)}} H^1\left(X, K_X^{\otimes 2}(-D)\right)
\end{equation}
are identified.

Let $\mathcal G$ (respectively $\mathcal G^{\rel}$) be the subsheaf of $T_{\mathcal X_g^{\BPS}(\kappa)}$ (respectively $T_{\mathcal X_g^{\BPS}(\kappa)/\mathcal P_g(\kappa)}$) of vector fields preserving the relative branched projective structure $\mathfrak p_g(n)$.
The Kodaira-Spencer exact sequence for the family $\mathcal X^{\BPS}_g(\kappa)$, when restricted to vector fields preserving $\mathfrak p_g(n)$, is given by :
\begin{equation}
0 \to \mathcal G^{\rel} \to \mathcal G \to \pi_g^{\BPS}(\kappa)^* T_{\mathcal P_g(\kappa)} \to 0
\end{equation}

By applying the left exact functor $\pi_g^{\BPS}(\kappa)_*$, one gets a morphism of sheaves $T_{\mathcal P_g(\kappa)} \to R^1 \pi_g^{\BPS}(\kappa)_* \mathcal G^{\rel}$.
In particular, at the point $(X, D, a, p, m)$, one gets a linear map
\begin{equation}\label{eq_kodaira_spencer_4}
KS^{\BPS} : \left(T_{\mathcal P_g(\kappa)}\right)_{(X, D, a, p, m)} \to H^1\left(X, \Lambda_p\right)
\end{equation}

Since $n > 4g-4$, $\left(T_{P_X^a}\right)_p = H^0\left(X, K_X^{\otimes 2}(-D)\right) = 0$.
Thus one has the following diagram, that is easily seen to be commutative :

\begin{equation}
\begin{tikzcd}
 0 \arrow[r] & \left(T_{\mathcal P_g(\kappa)}\right)_{(X, D, a, p, m)} \arrow[d, "KS^{\BPS}"] \arrow[r] & \left(T_{\mathcal A_g(\kappa)}\right)_{(X, D, a, m)} \arrow[d, "KS^{\BC}"] \arrow[r, "d\mathfrak c_g(\kappa)"] & H^1\left(X, K_X^{\otimes 2}(-D)\right) \arrow[d, equal]& \\ 0  \arrow [r] & H^1\left(X, \Lambda_p\right) \arrow[r] & H^1\left(X, \Gamma_a\right) \arrow[r]& H^1\left(X, K_X^{\otimes 2}(-D)\right) \arrow[r]& H^2\left(X, \Lambda_p\right)
\end{tikzcd}
\end{equation}

Since $KS^{\BC}$ is an isomorphism, one has :

\begin{lemma}
The map $KS^{\BPS}$ in \eqref{eq_kodaira_spencer_4} is an isomorphism.
Moreover the differential $d\mathfrak c_g(\kappa)$ is surjective at any point $p$ of $\mathcal P_g(\kappa)$ such that $H^2\left(X, \Lambda_p\right)=0$.
\end{lemma}

Let $p \in \mathcal P_g(\kappa)$ be such that $H^2\left(X, \Lambda_p\right) = 0$.
Since the restricted differential at $d\mathfrak c_g(\kappa)$ is surjective at $p$, the total differential $d\mathfrak c_g(n)$  is \emph{a fortiori} surjective at $p$.
This implies that $p$ is a smooth point of $\mathcal P_g(n) \subset \mathcal A_g(n)$.
Reciprocally, if $H^2\left(X, \Lambda_p\right) \neq 0$ and $D$ is reduced, then $h^1\left(X, \Lambda_p\right) = 6g-6+n+h^2\left(X, \Lambda_p\right) > 6g-6+n$, so $p$ is singular.

Finally, let $p \in \mathcal P_g(n)$ be any branched projective structure such that $h^2\left(X, \Lambda_p\right) > 0$.
It is shown in \cite{calsamiglia-deroin-francaviglia} that $p$ is the limit in $\mathcal P_g(n)$ of a sequence $\left(p_n\right)_{n \in \N}$ of branched projective structures with the same holonomy as $p$ and reduced branching divisor.
By subsection \ref{subsection_infinitesimal_automorphisms}, the dimension of $H^2\left(X, \Lambda_p\right)$ depends only on the holonomy of $\rho$, thus the $p_n$ are all singular points of $\mathcal P_g(n)$, and since the singular locus is closed, $p$ is itself a singular point.
\end{proof}

\subsubsection{General Case}\label{subsubsection_singular_points_general_case}

In order to address the case where $2g-2 \le n \le 4g-4$, we need to have a closer look at the point of view on $\mathcal P_g(n)$ given in subsection \ref{subsection_families_rbps}, and to work with restricted branched projective structures.

Let $X$ be a Riemann surface of genus $g \ge 2$, $D$ and $D'$ divisors on $X$, $a \in A_X^D$ a branching class and $r \in P_{X, D'}^a$ a jet of branched projective structure.
Let $U$ be an open subset of $X$ and let $V \in\Gamma(U, \Gamma_a)$, where $\Gamma_a$ is defined in \eqref{definition_gamma_alpha}.
For any $t \in \R$ that is near $0$, the flow of $V$ at time $t$ is still denoted by $\phi^t_V$.
Write
\begin{equation}
\widetilde \xi_r(V) = \lim_{t \to 0} \frac{(\phi^t_V)^*r - r}t \in \Gamma\left(U, K_X^{\otimes 2}(-D)|_{D'}\right)
\end{equation}
According to equation \ref{eq_xi_p}, the morphism of sheaves $\widetilde \xi_r : \Gamma_a \to K_X^{\otimes 2}(-D)|_{D'}$ is surjective.
Denote by $\widetilde \Lambda_r$ its kernel, so that there is an exact sequence :

\begin{equation}
0 \to \widetilde \Lambda_r \to \Gamma_a \xrightarrow{\widetilde \xi_r} K_X^{\otimes 2}(-D)|_{D'} \to 0
\end{equation}
which gives, by taking cohomologies, an exact sequence :
\begin{equation}\label{eq_exact_cohom_classes_jets}
0 \to H^0\left(X, K_X^{\otimes 2}(-D)|_{D'}\right) \to H^1\left(X, \widetilde \Lambda_r\right) \to H^1(X, \Gamma_a) \to 0
\end{equation}

Fix $n \in \N$.
Let $\mathfrak D$ be a divisor on $\mathcal X_g^{\BC}(n)$ such that $\mathfrak D$ contains no fiber of the map $\pi_g^{\BC}(n): \mathcal X_g^{\BC}(n) \to \mathcal A_g(n)$.
Let $\mathfrak n$ be the degree of the restriction of $\mathfrak D$ to the fibers of $\mathcal X_g^{\BC}(n)$.
Choose $\mathfrak D$ such that $n + \mathfrak n > 4g-4$.

Applying lemma \ref{lemma_universal_family_rbps_over_bc}, let us write $\mathcal U_g\left(n, \mathfrak D\right) = \bigl(\mathcal P_g(n, \mathfrak D), \mathcal X_g(n, \mathfrak D), \pi_g(n, \mathfrak D), m_g(n, \mathfrak D), \mathcal D_g(n, \mathfrak D),\break \mathfrak a_g(n, \mathfrak D), \mathfrak D', \mathfrak r_g(n, \mathfrak D)\bigr)$ the universal family of restricted branched projective structures over the family of branching classes with auxiliary divisors $\left(\mathcal A_g(n), \mathcal X_g^{\BC}(n), \pi_g^{\BC}(n), m_g^{\BC}(n), \mathcal D_g^{\BC}(n), \mathfrak a_g^{\BC}(n), \mathfrak D\right)$.
Let $\overline \alpha_g(n, \mathfrak D) : \mathcal P_g(n, \mathfrak D) \to \mathcal A_g(n)$ be the forgetful morphism.
It is an affine bundle by subsection \ref{subsection_families_rbps}.

For any partition $\kappa$ of the integer $n$, denote by $\mathcal P_g(\kappa, \mathfrak D)$ the restriction of the bundle $\overline \alpha_g(n, \mathfrak D) : \mathcal P_g(n, \mathfrak D) \to \mathcal A_g(n)$ to the basis $\mathcal A_g(\kappa) \subset \mathcal A_g(n)$.
Let $\mathcal U_g(\kappa, \mathfrak D) = (\mathcal P_g(\kappa, \mathfrak D), \mathcal X_g(\kappa, \mathfrak D), \mathcal D_g(\kappa, \mathfrak D),\break \pi_g(\kappa, \mathfrak D), \mathfrak a_g(\kappa, \mathfrak D), \mathfrak D'(\kappa), \mathfrak r_g(\kappa, \mathfrak D))$ be the pullback by the inclusion $\mathcal P_g(\kappa, \mathfrak D) \to \mathcal P_g(n, \mathfrak D)$ of the family $\mathcal U_g(n, \mathfrak D)$.
Denote by $\overline \alpha_g(\kappa, \mathfrak D)$ the restriction of $\overline \alpha_g(n, \mathfrak D)$ to $\mathcal P_g(\kappa, \mathfrak D)$.

Let $r \in \mathcal P_g(n, \mathfrak D)$.
Let $(X, D, a, m) \in \mathcal A_g(n)$ be the marked curve with divisor of degree $n$ and branching class, image of $r$ by the map $\overline \alpha_g(n, \mathfrak D)$.
Let $D'$ be the divisor of degree $\mathfrak n$ on $X$, restriction of $\mathfrak D$ to the fiber of $\mathcal X_g^{\BC}(n)$ over $(X, D, a, m)$.
Denote by $\kappa$ the partition of $n$ associated to the divisor $D$.
The tangent space of the fiber $P_{X, D'}^a = \overline \alpha(\kappa, \mathfrak D)^{-1}(X, D, a, m)$ at $r$ is $\left(T_{P_{X, D'}^a}\right)_r = H^0(X, K_X^{\otimes 2}(-D)|_{D'})$, according to proposition \ref{prop_struct_space_jets_bps}.

Moreover we saw that there is a Kodaira-Spencer isomorphism $KS^{\BC} : \left(T_{\mathcal A_g(\kappa)}\right)_{(X, D, a, m)} \xrightarrow{\sim} \break H^1\left(X, \Gamma_a\right)$.
As a consequence, one has identifications between the kernels and cokernels of the exact sequences \eqref{eq_exact_cohom_classes_jets} and
\begin{equation}
0 \to \left(T_{P_{X, D'}^a}\right)_r \to \left(T_{\mathcal P_g(\kappa, \mathfrak D)}\right)_r \to \left(T_{\mathcal A_g(\kappa)}\right)_{(X, D, a, m)} \to 0
\end{equation}

Let $\mathcal G_{\mathfrak D}$ (respectively $\mathcal G_{\mathfrak D}^{\rel}$) be the subsheaf of $\mathcal T_{\mathcal X_g(\kappa, \mathfrak D)}$ (respectively $\mathcal T_{\mathcal X_g(\kappa, \mathfrak D)/\mathcal P_g(\kappa, \mathfrak D)}$) that contains the vector fields preserving the restricted relative projective structure $\mathfrak r_g(n, \mathfrak D)$.
The Kodaira-Spencer exact sequence for the family $\mathcal X_g(\kappa, \mathfrak D)$, when restricted to the vector fields preserving $\mathfrak r_g(\kappa, \mathfrak D)$, is given by :

\begin{equation}
0 \to \mathcal G_{\mathfrak D}^{\rel} \to \mathcal G_{\mathfrak D} \to \pi_g(\kappa, \mathfrak D)^* T_{\mathcal P_g(\kappa, \mathfrak D)} \to 0
\end{equation}
Applying the left exact functor $\pi_g(\kappa, \mathfrak D)_*$, one gets a morphism of sheaves $T_{\mathcal P_g(\kappa, \mathfrak D)} \to R^1 \pi_g(\kappa, \mathfrak D)_* \mathcal G_{\mathfrak D}^{\rel}$.
In particular, at the point $r$, this morphism gives a linear map

\begin{equation}\label{eq_kodaira_spencer_D}
KS_{\mathfrak D} : \left(T_{\mathcal P_g(\kappa, \mathfrak D)}\right)_r \to H^1\left(X, \widetilde \Lambda_r\right)
\end{equation}

It is easy to see that the following diagram is commutative :

\begin{equation}
\begin{tikzcd} 0 \arrow[r] &\left(T_{P_{X, D'}^a}\right)_r \arrow[d, equal] \arrow[r] & \left(T_{\mathcal P_g(\kappa, \mathfrak D)}\right)_r \arrow[d, "KS_{\mathfrak D}"] \arrow[r] & \left(T_{\mathcal A_g(\kappa)}\right)_{(X, D, a, m)} \arrow[d, "KS^{\BC}"] \arrow[r] & 0 \\ 0 \arrow[r] & H^0\left(X, K_X^{\otimes 2}(-D)|_{D'}\right) \arrow [r] & H^1\left(X, \widetilde \Lambda_r\right) \arrow[r] & H^1\left(X, \Gamma_a\right) \arrow[r]& 0 \end{tikzcd}
\end{equation}

Since $KS^{\BC}$ is an isomorphism, we have the following lemma :

\begin{lemma}
The map $KS_{\mathfrak D}$ in \eqref{eq_kodaira_spencer_D} is an isomorphism.
\end{lemma}

Since $n + \mathfrak n > 4g-4$, one has a vector bundle on $\mathcal P_g(n, \mathfrak D)$ whose fiber over a point $r \in P_{X, D'}^a$ is $H^1\left(X, K_X^{\otimes 2}(-D-D')\right)$. 
Denote by $\mathfrak c_g(n, \mathfrak D) = \overline {\mathfrak c}(\mathfrak a_g(n, \mathfrak D), \mathfrak r_g(n, \mathfrak D))$ its global section defined in section \ref{subsection_families_rbps}.
According to section \ref{subsection_families_rbps}, the space $\mathcal P_g(n)$ is an analytic subspace of $\mathcal P_g(n, \mathfrak D)$, namely the vanishing locus of the section $\mathfrak c_g(n, \mathfrak D)$.
In particular, for any $p \in \mathcal P_g(n)$, the differential $\left(d \mathfrak c_g(n, \mathfrak D) \right)_p$ is a linear map from $\left(T_{\mathcal P_g(n, \mathfrak D)}\right)_p$ to $H^1\left(X, K_X^{\otimes 2}\left(-D-D'\right)\right)$, where $X$ is the curve underlying $p$, $D$ is the branching divisor of $p$ and $D'$ is the auxiliary divisor.
We will show that this differential is surjective for any $p \in \mathcal P_g(n)$ such that $H^2\left(X, \Lambda_p\right)=0$, which implies smoothness of $\mathcal P_g(n)$ at $p$.

For any partition $\kappa$ of the integer $n$, $\mathcal P_g(\kappa)$ is the intersection of the subspace $\mathcal P_g(n) \subset \mathcal P_g(n, \mathfrak D)$ with $\mathcal P_g(\kappa, \mathfrak D) \subset \mathcal P_g(n, \mathfrak D)$.
Recall that the family $\mathcal U_g^{\BPS}(\kappa)$ is the pullback of the family $\mathcal U_g^{\BPS}(n)$  by the inclusion $\mathcal P_g(\kappa) \to \mathcal P_g(n)$.

Let $(X, D, a, p, m) \in \mathcal P_g(n)$ be a marked curve with divisor of degree $n$, branching class and branched projective structure.
Let $\kappa$ be the partition of $n$ associated to the divisor $D$.
Denote by $r$ the branched projective structure $p$, seen as an element of $\mathcal P_g(\kappa, \mathfrak D)$.

Denote by $\mathcal G$ (respectively $\mathcal G^{\rel}$) the subsheaf of $T_{\mathcal X_g^{\BPS}(\kappa)}$ (respectively $T_{\mathcal X_g^{\BPS}(\kappa)/\mathcal P_g(\kappa)}$) containing vector fields preserving the relative branched projective structure $\mathfrak p_g(\kappa)$.
The Kodaira-Spencer exact sequence for the family $\mathcal X_g(\kappa, \mathfrak D)$, when restricted to vector fields preserving $\mathfrak p_g(\kappa)$, is given by :
\begin{equation}
0 \to \mathcal G^{\rel} \to \mathcal G \to \pi_g^{\BPS}(\kappa)^* T_{\mathcal P_g(\kappa)} \to 0
\end{equation}
By applying the left exact functor $\pi_g^{\BPS}(\kappa)_*$, one gets a morphism of sheaves $T_{\mathcal P_g(\kappa)} \to R^1 \pi_g^{\BPS}(\kappa)_* \mathcal G^{\rel}$.
At the point $(X, D, a, p, m)$, one gets a linear map
\begin{equation}\label{eq_kodaira_spencer_3bis}
KS^{\BPS} : \left(T_{\mathcal P_g(\kappa)}\right)_{(X, D, a, p, m)} \to H^1\left(X, \Lambda_p\right)
\end{equation}

It is easy to see that the following diagram is commutative, where $D'$ is the auxiliary divisor for the branching class $a$ and $r$ is $p$ restricted to $D'$ :

\begin{equation}
\begin{tikzcd}[column sep = small]
 0 \arrow[r] & \left(T_{\mathcal P_g(\kappa)}\right)_{(X, D, a, p, m)} \arrow[d, "KS^{\BPS}"] \arrow[r] & \left(T_{\mathcal P_g(\kappa, \mathfrak D)}\right)_r \arrow[d, "KS_{\mathfrak D}"] \arrow[r, "\left(d\mathfrak c_g(\kappa\comma\mathfrak D)\right)_r"] &[3em] H^1\left(X, K_X^{\otimes 2}\left(-D-D'\right)\right) \arrow[d, equal]& \\ 0  \arrow [r] & H^1\left(X, \Lambda_{p}\right) \arrow[r] & H^1\left(X, \widetilde \Lambda_r\right) \arrow[r]& H^1\left(X, K_X^{\otimes 2}\left(-D-D'\right)\right) \arrow[r]& H^2\left(X, \Lambda_p\right)
\end{tikzcd}
\end{equation}

Since $KS_{\mathfrak D}$ is an isomorphism, one has :

\begin{lemma}
The map $KS^{\BPS}$ in \eqref{eq_kodaira_spencer_3bis} is an isomorphism.
Moreover the differential $d\mathfrak c_g(\kappa, \mathfrak D)$ is surjective at any point $p$ of $\mathcal P_g(\kappa)$ such that $H^2\left(X, \Lambda_p\right)=0$.
\end{lemma}

Take $p \in \mathcal P_g(\kappa)$ such that $H^2\left(X, \Lambda_p\right)=0$.
Since the restricted differential $d\mathfrak c_g(\kappa, \mathfrak D)$ is surjective at $p$, the total differential $d\mathfrak c_g(n, \mathfrak D)$  is \emph{a fortiori} surjective at $p$.
This implies that $p$ is a smooth point of $\mathcal P_g(n) \subset \mathcal P_g(n, \mathfrak D)$.
Reciprocally, if $H^2\left(X, \Lambda_p\right) \neq 0$ and $D$ is reduced, then $h^1\left(X, \Lambda_p\right) = 6g-6+n+h^2\left(X, \Lambda_p\right) > 6g-6+n$, so $p$ is singular.

Finally, let $p \in \mathcal P_g(n)$ be any branched projective structure such that $h^2\left(X, \Lambda_p\right) > 0$.
It is shown in \cite{calsamiglia-deroin-francaviglia} that $p$ is the limit in $\mathcal P_g(n)$ of a sequence $\left(p_n\right)_{n \in \N}$ of branched projective structures with the same holonomy as $p$ and reduced branching divisor.
By subsection \ref{subsection_infinitesimal_automorphisms}, the dimension of $H^2\left(X, \Lambda_p\right)$ depends only on the holonomy of $\rho$, thus the $p_n$ are all singular points of $\mathcal P_g(n)$, and since the singular locus is closed, $p$ is itself a singular point.

\subsection{The Holonomy Map}

Let $S$ be a surface of genus $g$ (with $g \ge 2$).
Denote by $\Hom^{\star}\left(\pi_1(S), \PSL(2, \C)\right)$ the space of group morphisms $\rho : \pi_1(S) \to \PSL(2, \C)$ such that $\Image \rho \subset \PSL(2, \C)$ is not abelian and is not conjugated to a subgroup of the group $\left\{z \mapsto az| a \in \C^*\right\} \cup \left\{z \mapsto \frac az | a \in \C^*\right\}$ \footnote{This second condition is necessary for the action by conjugation of $\PSL(2, \C)$ to be free}.
In particular, $\Hom^{\star}\left(\pi_1(S), \PSL(2, \C)\right)$ contains all non-elementary representations.
The space $\Hom^{\star}\left(\pi_1(S), \PSL(2, \C)\right)$ is an open subspace of $\Hom\left(\pi_1(S), \PSL(2, \C)\right)$.
The space $\Hom^{\star}\left(\pi_1(S), \PSL(2, \C)\right)$ is a smooth analytic space (see \cite{hubbard}) and the group $\PSL(2, \C)$ acts freely and properly by conjugation on $\Hom^{\star}\left(\pi_1(S), \PSL(2, \C)\right)$.
As a consequence, the quotient $\Hom^{\star}\left(\pi_1(S), \PSL(2, \C)\right) / \PSL(2, \C)$ is a smooth analytic space.

Denote by $\mathcal P_g(n)^{\star}$ the subset of $\mathcal P_g(n)$ of branched projective structures whose holonomy lies in $\Hom^{\star}\left(\pi_1(S), \PSL(2, \C)\right)$.
There is a map $\Hol_g(n)^{\star} : \mathcal P_g(n)^{\star} \to \Hom^{\star}\left(\pi_1(S), \PSL(2, \C)\right)/\PSL(2, \C)$ that to a branched projective structure associates the conjugacy class of its holonomy representations.
The space $\mathcal P_g(n)^{\star}$ is clearly an open subset of $\mathcal P_g(n)$, and it is smooth by theorem \ref{theorem_singular_points}.
Since $\mathcal P_g(n)^{\star}$ is the basis of a family of branched projective structures, the map $\Hol_g(n)^{\star}$ is holomorphic.

\begin{theorem}\label{theorem_holonomy}
The holonomy map $\Hol_g(n)^{\star} : \mathcal P_g(n)^{\star} \to \Hom^{\star}\left(\pi_1(S), \PSL(2, \C)\right)/PSL(2, \C)$ is a holomorphic submersion.
\end{theorem}

Let $\mathcal X_g$ be the $\PSL(2, \C)$-character variety associated to $S$, i.e. the GIT quotient\break $\Hom\left(\pi_1(S), \PSL(2, \C)\right) // \PSL(2, \C)$, that is an analytic space.
There is an analytic morphism $\Hol_g(n) : \mathcal P_g(n) \to \mathcal X_g$, that to a branched projective structure associates its holonomy.
There is an obvious map $\iota : \Hom^{\star}\left(\pi_1(S), \PSL(2, \C)\right)/\PSL(2, \C) \to \mathcal X_g$, that verifies in particular $\Hol_g(n) = \iota \circ \Hol_g(n)^*$.
In restriction to the preimage of the smooth locus of $\mathcal X_g$, the map $\iota$ is a biholomorphism.
As a consequence, theorem \ref{theorem_holonomy} implies :

\begin{corollary}
In restriction to the preimage of the smooth locus of the character variety $\mathcal X_g$, the holonomy map $\Hol_g(n) : \mathcal P_g(n) \to \mathcal X_g$ is a holomorphic submersion.
\end{corollary}

\begin{proof}(of theorem \ref{theorem_holonomy})

Take $p \in \mathcal P_g(n)^*$, with underlying Riemann surface $X$.
Let $\rho \in \Hom^{\star}\left(\pi_1(S), \PSL(2, \C)\right)$ be a holonomy representation of $p$, and write $[\rho] = \Hol_g(n)^{\star}(p)$.
It is proved in \cite{hubbard} that the tangent space of $\Hom^{\star}\left(\pi_1(S), \PSL(2, \C)\right)/\PSL(2, \C)$ at $[\rho]$ is $H^1(X, \Ad_p)$, where $\Ad_p$ is defined in  subsection \ref{subsection_infinitesimal_automorphisms}.

Hubbard also proved in \cite{hubbard} that in the case $n = 0$, the differential at $p$ of the holonomy map, i.e. $d_p\Hol_g(0)^{\star} : T_p \mathcal P_g(0) \to H^1(X, \Ad_p)$, is given by the Kodaira-Spencer at $p$ of the universal family $\mathcal U_g^{bps}(0)$.
The situation is almost the same in the branched case $(n > 1)$, with two differences. Firstly the Kodaira-Spencer map $KS^{bps}$ (see subsection \ref{subsection_tangent_spaces_univ_spaces_bps}) is only defined on $T_p\mathcal P_g(\kappa)$, where $\kappa$ is the partition of $n$ associated to the branching divisor of $p$. Secondly $KS^{bps}$ takes values in the first cohomology group of the subsheaf $\Lambda_p$ of $\Ad_p$ (see subsection \ref{subsection_infinitesimal_automorphisms}). However there is a canonical map $\phi : H^1(X, \Lambda_p) \to H^1(X, \Ad_p)$, coming from the inclusion $\Lambda_p \subset \Ad_p$.

Thus the proof of Hubbard gives in the branched case :
\begin{equation}
d_p\Hol_g(n)^{\star}|_{T_p\mathcal P_g(\kappa)} = \phi \circ KS^{bps}
\end{equation}
Since $\phi$ is surjective according to \eqref{eq_exact_autom_local_syst} and $KS^{bps}$ is an isomorphism, $d_p\Hol_g(n)^{\star}|_{T_p\mathcal P_g(\kappa)}$ is surjective, thus $d_p\Hol_g(n)^{\star}$ is \emph{a fortiori} surjective, which proves that $\Hol_g(n)^{\star}$ is a submersion.
\end{proof}

\printbibliography

\end{document}